\documentclass[a4paper,oneside,10pt]{article}%
\usepackage[english]{babel}
\usepackage[T1]{fontenc}
\usepackage[a4paper,top=3cm,bottom=3cm,left=3cm,right=3cm,marginparwidth=1.75cm]%
{geometry}
\usepackage[colorlinks=true, allcolors=blue, hypertexnames=false]{hyperref}
\usepackage{empheq}
\usepackage{graphicx}
\usepackage[colorinlistoftodos]{todonotes}
\usepackage{amsmath}
\usepackage{mathtools}
\usepackage{upgreek}
\usepackage{amssymb}
\usepackage{amsfonts}
\usepackage{amsthm}
\usepackage{hyperref}
\usepackage{enumitem}
\usepackage{mathrsfs}
\usepackage{fontenc}
\usepackage{verbatim}
\usepackage{framed}
\usepackage{algorithm}
\usepackage{bbm}
\usepackage{algpseudocode}
\usepackage{appendix}
\usepackage{color}
\usepackage{multirow}%
\providecommand{\U}[1]{\protect\rule{.1in}{.1in}}
\newtheorem{theorem}{Theorem}[section]

\newtheorem{corollary}[theorem]{Corollary}

\newtheorem{assumption}[theorem]{Assumption}

\newtheorem{lemma}[theorem]{Lemma}

\newtheorem{proposition}[theorem]{Proposition}
\newtheorem{remark}[theorem]{Remark}

\numberwithin{equation}{section}

\begin{document}

\title{A Stochastic Maximum Principle for Forward-Backward Stochastic
Control Systems with Quadratic Generators and Sample-wise Constraints}
\author{Shaolin Ji\thanks{Zhongtai Securities Institute for Financial Studies,
Shandong University, 250100 Jinan, China. Email: jsl@sdu.edu.cn 
Research supported by the National Natural Science
Foundation of China (No. 11971263; 11871458).}
\and Rundong Xu\thanks{School of Mathematical Science,
Fudan University, 200433 Shanghai, China. Email:
rundong\_xu@fudan.edu.cn (Corresponding author)} }
\maketitle

\textbf{Abstract}. This paper examines the stochastic maximum principle (SMP) for a forward-backward stochastic
control system where the backward state equation is characterized by the backward stochastic differential equation (BSDE) with quadratic growth 
and the forward state at the terminal time is constrained in a convex set with probability one. 
With the help of the theory of BSDEs with quadratic growth and the bounded mean oscillation (BMO) martingales, 
we employ the terminal perturbation approach and Ekeland's variational principle to obtain a dynamic stochastic maximum principle.
The main result has a wide range of applications in mathematical finance and we investigate a robust recursive utility maximization problem
with bankruptcy prohibition as an example.



\addcontentsline{toc}{section}{\hspace*{1.8em}Abstract}

\section{Introduction}

The class of backward stochastic differential equations (BSDEs), with generators having a quadratic growth in the state variable $z$, has attracted much attention in the past two decades.
Besides the increasingly developed and enriched existence and uniqueness theory \cite{Barr-ElKar, HuBSDEquad08, Delbaen-Hu-Richou, Hu-Tang2016, Kobylanski, Tevzadze, Xing18}, BSDEs with quadratic growth have found
applications in stochastic control and mathematical finance, 
say, stochastic linear-quadratic control with random coefficients \cite{Bismut}, utility maximization problems \cite{Delbaen02, Hu-Im-Mull} etc.

%The objective is to minimize the cost functional $J(u(\cdot)):=\mathbb{E}\left[h(X_{T}^{u},Y_{0}^{u})\right]$ introduced in \cite{Yong2010} over all admissible controls $u(\cdot)$, where $h$ is any given smooth function.

In this paper, motivated especially by their applications in the risk-sensitive optimal control problems \cite{Karoui-Hamadene, Lim-Zhou, Moon2021} and 
the robust portfolio-consumption optimization model under model uncertainty \cite{Skiadas-2001} together with the related recursive utility maximization problems \cite{Matoussi2011} with the bankruptcy prohibition,
we are encountering with the following stochastic recursive optimal control problems involving BSDEs with quadratic growth and state constraints with sample-wise type 
(a sample-wise constraint requires that the state at certain time or at all times be in a prescribed set with probability $1$). 
Denote by $\mathcal{U}[0,T]$ the set of all the admissible controls and the cost functional is defined by the following mixed initial-terminal type (see \cite{Yong2010})
\begin{equation}
    \begin{array}
    [c]{rl}%
    J(u(\cdot)):=\mathbb{E}\left[h(X_{T}^{u},Y_{0}^{u})\right],
    \end{array}
    \label{intro-cost-func}
\end{equation}
where $h$ is any given smooth function, and $X^{u}(\cdot)$, $Y^{u}(\cdot)$ are the solutions to the controlled forward-backward stochastic differential equation (FBSDE, see \cite{HuYing-Peng95, Ma-Yong-Protter, Ma-WZZ})
\begin{equation}
    \left\{
    \begin{array}
    [c]{rl}%
    dX_{t}^{u}= & b(t,X_{t}^{u},u_{t})dt+\sigma(t,X_{t}^{u},u_{t})dW_{t},\\
    dY_{t}^{u}= & -f(t,X_{t}^{u},Y_{t}^{u},Z_{t}^{u},u_{t})dt+\left(  Z_{t}%
    ^{u}\right)  ^{\intercal}dW_{t},\text{  }t\in [0,T],\\
    X_{0}^{u}= & x_{0},\ Y_{T}^{u}=\Phi(X_{T}^{u}),
    \end{array}
    \right.
    \label{intro-FBSDE}
\end{equation}
where $W$ is a standard Brownian motion, the coefficients $b$, $\sigma$, $f$, $\Phi$ are deterministic, measurable functions in suitable sizes, and $f$ is quadratic growth in $z$.
The objective is to find $\bar{u}(\cdot) \in \mathcal{U}[0,T]$ (if it ever exists) such that
\begin{equation}
    \begin{array}
    [c]{rl}%
    J(\bar{u}(\cdot)) := \inf\limits_{u(\cdot) \in \mathcal{U}[0,T]} J(u(\cdot)),
    \end{array}
    \label{intro-objective}
\end{equation}
and the terminal state $X^{u}_{T}$ of the stochastic differential equation (SDE) in (\ref{intro-FBSDE}) is required to take values in a given convex set $K \subseteq \mathbb{R}^{n}$ ($n\in \mathbb{N}_{+}$) with probability one.
On the one hand, when $K=\mathbb{R}^{n}$, $h(x,y)=y$, $f(t,x,y,z,u)=\frac{\gamma}{2}\left\vert z \right\vert^{2} + g(t,x,u)$ with some $\gamma>0$ and measurable function $g$, if the coefficients admit enough integrability then 
(\ref{intro-FBSDE})-(\ref{intro-objective}) is closely related to the classical risk-sensitive control problems \cite{Karoui-Hamadene, Lim-Zhou} by using
exponential transformation and It\^{o}'s formula. After that Moon \cite{Moon2021} studied the generalized case if $g$ depends on $(y,z)$ through the dynamic programming approach.
On the other hand, under the setting of Brownian filtration and for any given $u(\cdot) \in \mathcal{U}[0,T]$, when $n=1$, $K=[0, +\infty)$, $h(x,y)=-y$, $f(t,x,y,z,u)=U(u) - \beta y - \frac{1}{2\theta}\left\vert z \right\vert^{2} $, and if $\mathcal{U}[0,T]$ represents the set of consumption-portfolio strategies $u(\cdot)$ feasible for the initial wealth $x_{0} \geq 0$,
it follows from the main result in \cite{Skiadas-2001} that $-J(u(\cdot))$ is optimal for the minimization part of the $\mathrm{sup}$-$\mathrm{inf}$ problem proposed in \cite{Matoussi2011} thanks to the method of dual representation (see also [Quenez03]), 
where $\theta$ is the risk-averse parameter and $U$ is the utility function. 
Furthermore, this means that the objective (\ref{intro-objective}) is equivalent to the maximization part of the sup-inf problem in \cite{Matoussi2011} .

The existence of constraints with sample-wise type as above is more a rule than an exception in reality, for example, 
in the continuous-time mean-variance portfolio selection problem \cite{Pliska} and the recursive utility maximization problems \cite{Karoui-Peng-Quenez-2001} with bankruptcy prohibition. 
Another important example is the study of the Neyman-Pearson lemma for hypothesis tests under a class of nonlinear probability measures---$g$-probabilities \cite{Ji-Zhou-2010}, 
where the setting $K=[0,1]$ plays a role as a criterion to exclude the tests that make the $g$-probability of Type I error beyond the given acceptable significance level.

Based on the above motivation, we are aimed at deducing the necessary condition of the optimality---stochastic maximum principle (SMP) for problem (\ref{intro-FBSDE})-(\ref{intro-objective}).
%It is well known that SMP is an important tool to investigate stochastic control problems.
Since Peng \cite{Peng90} obtained the general SMP for the classical stochastic control systems, 
%it has attracted much attention to studying the necessary condition of the optimality for various stochastic control systems (see \cite{Fuhrman-Hu06, Fuhrman-Hu13, Peng93, Tang-Li94} and references therein).
researchers have made progress in the SMP for coupled forward-backward stochastic control systems (see \cite{HuJiXue18, Peng93, Wu2013, Yong2010} and references therein) driven by FBSDEs when $K = \mathbb{R}^{n}$. 
For the case $K \subsetneq \mathbb{R}^{n}$, (\ref{intro-FBSDE})-(\ref{intro-objective}) is well studied \cite{Ji-Peng-2008, Ji-Zhou-2006} when $f$ is globally Lipschitz continuous in $(x,y,z)$
and it is generalized to the fully coupled case \cite{Wei2013} and to the mean-field case \cite{Wei2016}. In the existing literature, there are two main approaches to getting the SMP, one is based on the pure BSDE approach \cite{Karoui-Peng-Quenez-2001} and 
the other is based on the Ekeland variational principle \cite{Ji-Peng-2008, Ji-Zhou-2006, Wei2013}. Particularly, adopting the BSDE method, the authors in \cite{Matoussi2011} establish a comparison theorem for specific BSDEs with quadratic growth and derive a dynamic SMP in the semimartingale context.
However, the comparison theorem may not hold since we do not require $f$ have special structures or convexity/concavity in $z$, and therefore we will resort to the Ekeland variational principle to achieve this goal under our framework.

We encountered three difficulties in deducing the SMP for (\ref{intro-FBSDE})-(\ref{intro-objective}).
The first one is that the BSDE in (\ref{intro-FBSDE}) is no longer Lipschitz but quadratic growth in $z$, which leads to the derivative $f_{z}$ being unbounded.
The unboundedness of $f_{z}$ brings much trouble in obtaining the following estimate, for example, when $f$ depends only on $z$,
\begin{equation}
\label{intro-est}
\lim_{\varepsilon\rightarrow0}\mathbb{E}\left[  \left(  \int_{0}^{T}\left\vert
\int_{0}^{1}f_{z}(\bar{Z}_{t}+\lambda(Z_{t}^{\varepsilon}-\bar{Z}_{t}))d\lambda-f_{z}(\bar{Z}_{t})\right\vert ^{2}dt\right)  ^{p}\right]
=0
\end{equation}
for some $p>1$, where $\bar{Z}(\cdot)$ represents the optimal trajectory 
and $Z^{\varepsilon}(\cdot)$ represents the state trajectory after perturbation, because one can deduce (\ref{intro-est}) when $f$ is Lipschitz in $z$ by using the dominated convergence theorem.
The second one is that when the family of approximate controls produced by Ekeland's variational principle converges to the optimal one, in which appropriate space can we obtain the convergence of the solutions of the approximate variational equations to the one solving the original variational equation?
In the classical case, such an issue can be solved by applying the continuous dependence of the solutions to the Lipschitz BSDEs on the parameters, but it entails estimating the difference between the solutions from two different linear BSDEs with unbounded coefficients under our framework. Furthermore, to this end, we first need to ensure that the approximate state trajectories converge to the optimal one, which essentially involves the convergence of solutions of a family of quadratic BSDEs.
The third one is that the adjoint equation is a linear SDE with unbounded coefficients due to the unboundedness of $f_{z}$. When we deduce the SMP, the solution to it will appear as a component of the integrands of stochastic integrals with respect to the Brownian motion (see (\ref{Ito-lemma-applied})-(\ref{def-Gamma})). Such stochastic integrals are only local martingales whose mathematical expectation at the terminal time $T$ may not exist. So we must check all these stochastic integrals are true martingales on the time interval $[0, T]$ before taking the expectation.

To overcome the aforementioned difficulties, we deduce the desired convergence (\ref{intro-est}) by applying the energy-type inequality of the bounded mean oscillation (BMO) martingales together with the generalized dominated convergence theorem. Using the estimate in \cite{Briand-SL-BSDE} for the linear BSDEs with stochastic Lipschitz coefficients involving BMO martingales, the convergence of both the approximate state trajectories and the approximate variational equations are attained, and the former convergence is stronger than the latter one.
To tackle the last difficulty, we note that the solution of the adjoint equation is the Dol\'{e}ans-Dade exponential of a certain BMO martingale, which satisfies the reverse H\"{o}lder inequality ($R_{p}$) as long as $p \in [1,\bar{p})$ for some $\bar{p} >1$ (see \cite{Kazamaki}, Chapter 3, Definition 3.1). On the other hand, the most complicated term to estimate, in those integrands of stochastic integrals, is a product including the solutions of the variational equation and the adjoint equation. 
Based on this observation, we choose a proper $p \in (1,\bar{p})$ together with its conjugate $p^{\ast}$, such that the ($R_{p}$) condition holds and the solution of the variational equation admits a $p^{\ast}$-moment. Then we can apply H\"{o}lder's inequality with the couple $(p,p^{\ast})$ to the square root of the quadratic variation of that stochastic integral to verify it is a true martingale.

The rest of the paper is organized as follows. In section 2, preliminaries and the formulation of our problem are given. We use a pure backward formulation of (\ref{intro-FBSDE}) in which the terminal state $X_{T}^{u}$ is regarded as the control variable. 
Unlike the formulation in \cite{Ji-Peng-2008, Ji-Zhou-2010}, such a reformulated set of admissible controls no longer includes all square-integrable random variables but higher-order ones because of the quadratic growth in $z$. In section 3, employing the analysis of BMO martingales, 
we guarantee the well-posedness of both the variational equation and the adjoint equation. Then, applying Ekeland’s variational principle, we obtain a dynamic SMP that characterizes the optimal terminal state. 
In section 4, to illustrate the established SMP, we study its applications to a robust recursive utility maximization problem with bankruptcy prohibition.

\section{Preliminaries and problem formulation}

Let $T\in(0,+\infty)$, $n,d\in \mathbb{N}_{+}$ and $(\Omega,\mathcal{F},\mathbb{P})$ be a complete
probability space on which a standard $d$-dimensional Brownian motion
$W=(W_{t}^{1},W_{t}^{2},\ldots,W_{t}^{d})_{t\in\lbrack0,T]}^{\intercal}$ is
defined. $\mathbb{F}:\mathbb{=}\left\{  \mathcal{F}_{t},0\leq t\leq T\right\}
$ is the $\mathbb{P}$-augmentation of the natural filtration of $W$. Denote by
$\mathbb{R}^{n}$ the $n$-dimensional real Euclidean space and $\mathbb{R}%
^{n\times d}$ the set of $n\times d$ real matrices. The scalar
product (resp. norm) of any two $n\times d$ matrices $A$, $B$ is
denoted by $\left\langle A,B\right\rangle =\mathrm{tr}\{AB^{\intercal}\}$
(resp.$\vert A\vert=\sqrt{\mathrm{tr}\left\{  AA^{\intercal}\right\}  }$),
where the superscript $^{\intercal}$ denotes the transpose of vectors or matrices.

For any given $p,q\geq1$, we introduce the following spaces and notation.

\noindent$L_{\mathcal{F}_{T}}^{p}(\Omega;\mathbb{R}^{n})$: the space of $\mathcal{F}%
_{T}$-measurable, $\mathbb{R}^{n}$-valued random variables $\xi$ such that
$\left\Vert \xi \right\Vert_{L^{p}}:=\left(\mathbb{E}\left[  \left\vert \xi\right\vert ^{p}\right]\right)^{\frac{1}{p}}  <\infty$.

\noindent$L_{\mathcal{F}_{T}}^{\infty}(\Omega;\mathbb{R}^{n})$: the space of
$\mathcal{F}_{T}$-measurable, $\mathbb{R}^{n}$-valued random variables $\xi$
such that $\mathbb{P}-\underset{\omega\in\Omega}{\mathrm{ess~sup}}\left\vert \xi\left(
\omega\right)  \right\vert <\infty$.

\noindent$\mathcal{M}_{\mathcal{F}}^{p,q}([0,T];\mathbb{R}^{n})$: the space of
$\mathbb{F}$-adapted, $\mathbb{R}^{n}$-valued processes $\varphi(\cdot)$ on
$[0,T]$ such that
\[
\left\Vert \varphi(\cdot)\right\Vert _{p,q}=\left(  \mathbb{E}\left[  \left(
\int_{0}^{T}\left\vert \varphi_{t}\right\vert ^{p}dt\right)  ^{\frac{q}{p}%
}\right]  \right)  ^{\frac{1}{q}}<\infty.
\]
In particular, we denote by $\mathcal{M}_{\mathcal{F}}^{p}([0,T];\mathbb{R}%
^{n})$ the above space when $p=q$.

\noindent$L_{\mathcal{F}}^{\infty}([0,T];\mathbb{R}^{n})$: the space of $\mathbb{F}%
$-adapted, $\mathbb{R}^{n}$-valued\ processes $\varphi(\cdot)$ on $[0,T]$ such that
\[
\left\Vert \varphi(\cdot)\right\Vert _{\infty}=\lambda \otimes \mathbb{P}-\underset{(t,\omega)\in
\lbrack0,T]\times\Omega}{\mathrm{ess~sup}}\left\vert \varphi_{t}\left(
\omega\right)  \right\vert <\infty,
\]
where $\lambda$ denotes the Lebesgue measure on $[0,T]$.

\noindent$\mathcal{S}_{\mathcal{F}}^{p}([0,T];\mathbb{R}^{n})$: the space of continuous
processes $\varphi(\cdot)\in$ $\mathcal{M}_{\mathcal{F}}^{p}([0,T];\mathbb{R}%
^{n})$ such that
\[
\left\Vert \varphi \right\Vert_{\mathcal{S}^{p}}:=
\left(\mathbb{E}\left[\sup\limits_{t\in[0,T]}\left\vert \varphi_{t}\right\vert ^{p}\right]\right)^{\frac{1}{p}}
<\infty.
\]

\noindent$\mathrm{BMO}_{p}$: the space of real-valued, continuous $\mathbb{F}$-martingales $M$
with $M_{0}=0$ such that
\begin{equation}
\left\Vert M\right\Vert _{\mathrm{BMO}_{p}}:=\sup_{\tau}\left\Vert \left(
\mathbb{E}\left[  \left\vert M_{T}-M_{\tau}\right\vert ^{p}\mid\mathcal{F}%
_{\tau}\right]  \right)  ^{\frac{1}{p}}\right\Vert _{\infty}<\infty,\text{
\ }p\in\lbrack1,+\infty), \label{def-BMOp}%
\end{equation}
where the supremum is taken over all stopping times $\tau\in\mathcal{[}%
0,T\mathcal{]}$. By Corollary 2.1 in \cite{Kazamaki}, $M$ is a $\mathrm{BMO}_{p}$ martingale if and only if it is a $\mathrm{BMO}_{q}$ martingale for every $q\geq1$. Therefore, it is simple to write $\mathrm{BMO}$ to represent $\mathrm{BMO}_{p}$.

\noindent$\mathcal{E}\left(  M\right)  $: the Dol\'{e}ans-Dade exponential of a continuous
local martingale $M$, that is, $\mathcal{E}\left(  M_{t}\right)  =\exp\left\{
M_{t}-\frac{1}{2}\left\langle M\right\rangle _{t}\right\}  $ for any
$t\in\lbrack0,T]$.

\noindent$p_{M}^{\ast}$: the conjugate exponent of $p_{M}^{{}}$, i.e. $\left(
p_{M}^{{}}\right)  ^{-1}+\left(  p_{M}^{\ast}\right)  ^{-1}=1$, where $M \in \mathrm{BMO}$,
$p_{M}^{{}}$ is the positive constant defined by the following function:%
\begin{equation}
\Psi(x)=\sqrt{1+\frac{1}{x^{2}}\ln\frac{2x-1}{2(x-1)}}-1\text{\ for }%
x\in(1,+\infty) \label{Rp-cd-func}%
\end{equation}
with $\Psi(p_{M}^{{}})=\left\Vert M\right\Vert _{\mathrm{BMO}_{2}}$. By Theorem 3.1 in \cite{Kazamaki},
if $p\in\left(
1,p_{M}^{{}}\right)  $, then, for any stopping time $\tau\in\mathcal{[}%
0,T\mathcal{]}$,
\begin{equation}
\mathbb{E}[\left(  \mathbb{\mathcal{E}}\left(  M_{T}\right)  \right)
^{p}/\left(  \mathbb{\mathcal{E}}\left(  M_{\tau}\right)  \right)  ^{p}%
\mid\mathcal{F}_{\tau}]\leq C_{0} ,\text{ \ }\mathbb{P}\text{-a.s.}, \label{Rp-cd}%
\end{equation}
where $C_{0}$ is positive constant depending only on $p$ and $\mathrm{BMO}_{2}$, and
(\ref{Rp-cd}) is called the reverse H\"{o}lder inequality.

\noindent$H\cdot W$: $H$ is an $\mathbb{F}$-adapted process and $H\cdot W$ is the
stochastic integral of $H$ with respect to $W$. If $H\cdot W\in \mathrm{BMO}$, then we
write simply $p_{H}^{{}}$ for $p_{H\cdot W}^{{}}$ and $p_{H}^{\ast}$ for
$p_{H\cdot W}^{\ast}$ without ambiguity.

\subsection{Classical formulation}
Let $\bar{p}^{\ast}>1$ be a number which will be determined lately.
Consider the following forward-backward stochastic control system:
over the set of admissible controls
\[
\mathcal{U}[0,T]:=\{u(\cdot) \mid u(\cdot) \in \mathcal{M}_{\mathcal{F}}^{2,4\bar{p}^{\ast}}([0,T];\mathbb{R}^{n\times d})\},
\]
minimizing the cost functional
\begin{equation}
J(u(\cdot)):=\mathbb{E}\left[  h(X_{T}^{u},Y_{0}^{u})\right]  \label{obje-eq}%
\end{equation}
subject to the controlled FBSDE
\begin{equation}
\left\{
\begin{array}
[c]{rl}%
dX_{t}^{u}= & b(t,X_{t}^{u},u_{t})dt+\sigma(t,X_{t}^{u},u_{t})dW_{t},\\
dY_{t}^{u}= & -f(t,X_{t}^{u},Y_{t}^{u},Z_{t}^{u},u_{t})dt+\left( Z_{t}^{u}\right)^{\intercal} dW_{t},\\
X_{0}^{u}= & x_{0},\ Y_{T}^{u}=\Phi(X_{T}^{u}),
\end{array}
\right.  \label{state-eq}%
\end{equation}
and an additional convex constraint
\begin{equation}
\label{terminal-constraint}
X_{T} \in K,\ \ \mathbb{P}\text{-a.s.},
\end{equation}
where $K$ is a nonempty convex subset in $\mathbb{R}^{n}$, $x_{0} \in \mathbb{R}^{n}$,
\[%
\begin{array}
[c]{ll}%
b:[0,T]\times\mathbb{R}^{n}\times\mathbb{R}^{n\times d}\longmapsto
\mathbb{R}^{n}, & \sigma:[0,T]\times\mathbb{R}^{n}\times\mathbb{R}^{n\times
d}\longmapsto\mathbb{R}^{n\times d},\\
f:[0,T]\times\mathbb{R}^{n}\times\mathbb{R}\times\mathbb{R}^{d}\times
\mathbb{R}^{n\times d}\longmapsto\mathbb{R}, & \Phi:\mathbb{R}^{n}%
\longmapsto\mathbb{R},\\
h:\mathbb{R}^{n}\times\mathbb{R}\longmapsto\mathbb{R}, &
\end{array}
\]
are deterministic, measurable functions. We impose the following assumptions on the coefficients of (\ref{state-eq}).

\begin{assumption}
\label{assum-1}Let $L>0$ be given.

(i) $b$, $\sigma$, $f$, $h$, $\Phi$ are continuous in their arguments; $\Phi$ is
continuously differentiable; $b$, $\sigma$ are continuously differentiable in
$(x,u)$; $f$ is continuously differentiable in $(x,y,z,u)$; $h$ is continuously differentiable in $(x,y)$.

(ii) $\Phi$, $\Phi_{x}$, $b_{x}$, $\sigma_{x}$, $b_{u}$, $\sigma_{u}$, $f_{y}$
are bounded; $h_{x}$, $h_{y}$ are bounded by $L(1+\left\vert x\right\vert
+\left\vert y\right\vert )$.

(iii)%
\[%
\begin{array}
[c]{ll}%
\left\vert f(t,x,0,0,u)\right\vert \leq L, & \left\vert f_{x}%
(t,x,y,z,u)\right\vert \leq L\left(  1+\left\vert y\right\vert +\left\vert
z\right\vert ^{2}\right)  ,\\
\left\vert f_{z}(t,x,y,z,u)\right\vert \leq L\left(  1+\left\vert z\right\vert
\right)  , & \left\vert f_{u}(t,x,y,z,u)\right\vert \leq L\left(  1+\left\vert
y\right\vert +\left\vert z\right\vert \right)  .
\end{array}
\]
\end{assumption}

Let $A=e^{6L\left\Vert \Phi\right\Vert _{\infty}+LT}\left[  \left\Vert
\Phi\right\Vert _{\infty}+LT+2\left(  L^{-2}+T\right)  \right]$
and $\bar{p}$ be the constant such that
\begin{equation}
\Psi(\bar{p})=\sqrt{3L^{2}\left(  T+2A\right)},
\label{p-bar}
\end{equation}
where the function $\Psi(\cdot)$ is defined by (\ref{Rp-cd-func}).
We assign the value $\bar{p}(\bar{p}-1)^{-1}$ to $\bar{p}^{\ast}$. Obviously, $\bar{p}^{\ast}$ is the conjugate of $\bar{p}$.

\subsection{Backward formulation}
In this subsection, we give an equivalent backward formulation of the above stochastic optimal control problem
(\ref{obje-eq})-(\ref{state-eq}). To do so we need an additional assumption:
\begin{assumption}
\label{assum-2}
There exists $\alpha>0$ such that
\[
\left \vert \sigma(t,x,u_{1})-\sigma(t,x,u_{2}) \right \vert \geq \alpha \left \vert u_{1}-u_{2} \right \vert
\]
for all $x\in \mathbb{R}^{n}$, $t\in [0,T]$ and $u_{1},u_{2} \in \mathbb{R}^{n \times d}$.
\end{assumption}
Note that Assumptions \ref{assum-1} and \ref{assum-2} imply the mapping $u\longmapsto \sigma(t,x,u)$ is a bijection from
$\mathbb{R}^{n \times d}$ onto itself for any $(t,x)$. Therefore, let $\theta=\sigma(t,x,u)$ and denote the inverse function by
$u=\tilde{\sigma}(t,x,\theta)$. Then system (\ref{state-eq}) can be rewritten as
\[
\left\{
\begin{array}
[c]{rl}%
dX_{t}= & -l(t,X_{t},\theta_{t})dt+\theta_{t}dW_{t},\\
dY_{t}= & -g(t,X_{t},Y_{t},Z_{t},\theta_{t})dt+\left( Z_{t} \right)^{\intercal} dW_{t},\\
X_{0}= & x_{0},\ Y_{T}=\Phi(X_{T}),\text{ \ }t\in\lbrack0,T],
\end{array}
\right.
\]
where $l(t,x,\theta)=-b(t,x,\tilde{\sigma}(t,x,\theta))$ and $g(t,x,y,z,\theta)=f(t,x,y,z,\tilde{\sigma}(t,x,\theta))$. Since $u\longmapsto \sigma(t,x,u)$ is a bijection, 
we may regard $\theta(\cdot)$ as the control variable. Due to the well-posedness of the BSDEs with Lipschitz generators, selecting $\theta(\cdot)$ is equivalent to selecting the terminal state $X_{T}$. Then we obtain the following purely backward control system:
\begin{equation}
\left\{
\begin{array}
[c]{rl}%
dX_{t}^{\xi}= & -l(t,X_{t}^{\xi},\theta_{t}^{\xi})dt+\theta_{t}^{\xi}dW_{t},\\
X_{T}^{u}= & \xi,\\
dY_{t}^{\xi}= & -g(t,X_{t}^{\xi},Y_{t}^{\xi},Z_{t}^{\xi},\theta_{t}^{\xi
})dt+\left( Z_{t}^{\xi} \right)^{\intercal} dW_{t},\\
Y_{T}^{u}= & \Phi(\xi),\text{ \ }t\in\lbrack0,T],
\end{array}
\right.  \label{state-eq-new}%
\end{equation}
where $\xi$ is the control variable to be chosen from
\[
\mathcal{U}_{ad}=\{\xi\in L_{\mathcal{F}_{T}}^{4\bar{p}^{\ast}}%
(\Omega;\mathbb{R}^{n}):\xi(\omega)\in K,\text{ }\mathbb{P}\text{-a.s. }\omega
\in\Omega\}.
\]
The equivalent cost functional is
\begin{equation}
J(\xi):=\mathbb{E}\left[  h(\xi,Y_{0}^{\xi})\right]  .
\label{cost-functional}%
\end{equation}
Thus, the original problem is equivalent to minimizing $J(\xi)$ over $\mathcal{U}_{ad}$, subject to the controlled system (\ref{state-eq-new}) and the initial constraint $X_{0}^{\xi}=x_{0}$.

\begin{remark}
According to the definitions of $l$, $g$ and Assumption \ref{assum-1}, one can verify that $l$ and $g$ satisfy similar conditions in Assumption \ref{assum-1}.
\end{remark}

From the existence result (Proposition 3) in \cite{HuBSDEquad08} and the uniqueness result (Lemma 2.1) in \cite{Hu-Tang2016}, we have:

\begin{theorem}
\label{state-eq-exist-th} Let Assumptions \ref{assum-1} and \ref{assum-2}
hold. Then, for any $\xi\in\mathcal{U}_{ad}$,
(\ref{state-eq}) admits a unique solution $(X^{\xi}\left(  \cdot\right)
,Y^{\xi}\left(  \cdot\right)  ,Z^{\xi}\left(  \cdot\right)  ,\theta^{\xi
}(\cdot))\in\mathcal{S}_{\mathcal{F}}^{4\bar{p}^{\ast}}([0,T];\mathbb{R}%
^{n})\times L_{\mathcal{F}}^{\infty}([0,T];\mathbb{R})\times\mathcal{M}%
_{\mathcal{F}}^{2}([0,T];\mathbb{R}^{d})\times$ $\mathcal{M}_{\mathcal{F}%
}^{2,4\bar{p}^{\ast}}([0,T];\mathbb{R}^{n\times d})$ such that $Z_{{}}^{\xi
}\cdot W\in\mathrm{BMO}$. Furthermore, we have the following estimate:%
\begin{equation}
\left\{
\begin{array}
[c]{l}%
\left\Vert X^{\xi}\right\Vert _{\mathcal{S}^{4\bar{p}^{\ast}}}^{4\bar{p}%
^{\ast}}+\left\Vert \theta^{\xi}\right\Vert _{2,4\bar{p}^{\ast}}^{4\bar
{p}^{\ast}}\leq C\mathbb{E}\left[  \left\vert \xi\right\vert ^{4\bar{p}^{\ast
}}+\left(  \int_{0}^{T}\left\vert l(t,0,0)\right\vert dt\right)  ^{4\bar
{p}^{\ast}}\right]  ,\\
\left\Vert Y^{\xi}\right\Vert _{\infty}+\left\Vert Z^{\xi}\cdot W\right\Vert
_{\mathrm{BMO}_{2}}^{2}<A,
\end{array}
\right.  \label{est-quad-qfbsde}%
\end{equation}
where $C$ depends on $T$, $\bar{p}^{\ast}$, $\left\Vert l_{x}\right\Vert
_{\infty}$, $\left\Vert l_{u}\right\Vert _{\infty}$.
\end{theorem}

\begin{corollary}
\label{Zu-bounded} From the energy-type inequality (\cite{Kazamaki}, page 26) and
the second inequality in (\ref{est-quad-qfbsde}), applying H\"{o}lder's inequality yields
\[
\sup_{\xi\in\mathcal{U}_{ad}}\mathbb{E}\left[  \left(
%TCIMACRO{\dint _{0}^{T}}%
%BeginExpansion
{\displaystyle\int_{0}^{T}}
%EndExpansion
\left\vert Z_{t}^{\xi}\right\vert ^{2}dt\right)  ^{p}\right]  <\left(  \left[
p\right]  +1\right)  !A^{2p},\ \ \forall p>0.
\]
\end{corollary}

\section{Stochastic maximum principle}
Applying Ekeland’s variational principle,
we derive the stochastic maximum principle for the optimization
problem (\ref{state-eq-new})-(\ref{cost-functional}) in this section.
The proposition below, which will be used frequently, follows from Corollary 9 and Theorem 10 in \cite{Briand-SL-BSDE}.

\begin{proposition}
\label{est-linear-prop} Assume $\lambda(\cdot) \in L_{\mathcal{F}}^{\infty}([0,T];\mathbb{R})$,
$\mu\cdot W\in \mathrm{BMO}$ and $\left(  \xi,\varphi(\cdot)\right)  \in
L_{\mathcal{F}_{T}}^{\beta_{0}}(\Omega;\mathbb{R})\times\mathcal{M}%
_{\mathcal{F}}^{1,\beta_{0}}([0,T];\mathbb{R})$ for some $\beta_{0}>p_{\mu
}^{\ast}$. Then there exists a unique solution $(Y(\cdot),Z(\cdot))\in\bigcap
_{1<\beta<\beta_{0}}\left(  \mathcal{S}_{\mathcal{F}}^{\beta}([0,T];\mathbb{R}%
)\times\mathcal{M}_{\mathcal{F}}^{2,\beta}([0,T];\mathbb{R}^{d})\right)  $ to
the following BSDE
\[
Y_{t}=\xi+\int_{t}^{T}\left(  \lambda_{s}Y_{s}+\mu_{s}^{\intercal}%
Z_{s}+\varphi_{s}\right)  ds-\int_{t}^{T}Z_{s}^{\intercal}dW_{s},\text{  }t \in [0,T].
\]
Moreover, for any $\beta\in\left(  1,\beta_{0}\right)  $, we have
\begin{equation}
\mathbb{E}\left[  \sup\limits_{t\in\lbrack0,T]}\left\vert Y_{t}\right\vert
^{\beta}+\left(  \int_{0}^{T}\left\vert Z_{t}\right\vert ^{2}dt\right)
^{\frac{\beta}{2}}\right]  \leq C\left(  \mathbb{E}\left[  \left\vert
\xi\right\vert ^{\beta_{0}}+\left(  \int_{0}^{T}\left\vert \varphi
_{t}\right\vert dt\right)  ^{\beta_{0}}\right]  \right)  ^{\frac{\beta}%
{\beta_{0}}}, \label{est-linear-data}%
\end{equation}
where $C>0$ depends on\ $\beta,$ $\beta_{0}$, $T$, $\left \Vert \lambda \right \Vert_{\infty}$, $\left\Vert
\mu\cdot W\right\Vert _{\mathrm{BMO}_{2}}$, and increases with respect to $\left\Vert
\mu\cdot W\right\Vert _{\mathrm{BMO}_{2}}$.
\end{proposition}

\subsection{Variational equation}
In this subsection, the constant $C$ will change from line to line in our proof.

For $\xi_{1}$, $\xi_{2} \in \mathcal{U}_{ad}$, define a metric in $\mathcal{U}_{ad}$ by
\[
d(\xi_{1},\xi_{2}):=\left(\mathbb{E}\left[ \left\vert \xi_{1}-\xi_{2} \right\vert^{4\bar{p}^{\ast}} \right] \right)^{\frac{1}{4\bar{p}^{\ast}}}.
\]
One can verify that $\left( \mathcal{U}_{ad},d(\cdot,\cdot) \right)$ is a complete metric space.
In fact, if $\left\{ \xi_{m} \right\}_{m=1}^{\infty}$ is a Cauchy sequence in $\left( \mathcal{U}_{ad},d(\cdot,\cdot) \right)$,
then we can find a subsequence $\left\{ \xi_{m_{k}} \right\}_{k=1}^{\infty}$ such that $d(\xi_{m_{k+1}},\xi_{m_{k}})<\frac{1}{2^{k}}$.
Set $A_{0}:=\Omega$,
\[
A_{i}:=\bigcup_{k=i}^{\infty}\left\{ \omega \in \Omega: \left\vert \xi_{m_{k+1}}(\omega)-\xi_{m_{k}}(\omega) \right\vert > 0 \right\}, \ \ i=1,2,\ldots.
\]
Choose $\xi:=\sum_{i=1}^{\infty}\xi_{m_{i}}\mathrm{1}_{A_{i-1}\setminus A_{i}}$, where $\mathrm{1}_{A}$ denotes the indicator of any $A \in \mathcal{F}$.
Then we can deduce $\lim_{k\rightarrow \infty} d(\xi_{m_{k}},\xi)=0$ and $\xi \in \mathcal{U}_{ad}$. Since $\left\{ \xi_{m} \right\}_{m=1}^{\infty}$ is a Cauchy sequence in $\left( \mathcal{U}_{ad},d(\cdot,\cdot) \right)$, we also have $\lim_{m\rightarrow \infty} d(\xi_{m},\xi)=0$ from which the completeness of $\left( \mathcal{U}_{ad},d(\cdot,\cdot) \right)$ follows.

Let $\bar{\xi}$ be optimal and $(\bar{X}(\cdot),\bar{Y}%
(\cdot),\bar{Z}(\cdot),\bar{\theta}(\cdot))$ be the corresponding state trajectory to (\ref{state-eq-new}).
For $i=1,2,\ldots
,n$, set
\[
l(\cdot)=\left(  l^{1}(\cdot),l^{2}(\cdot),\ldots,l_{{}}^{n}(\cdot)\right)
^{\intercal}\in\mathbb{R}^{n},
\]%
\[%
\begin{array}
[c]{cc}%
l_{x}(\cdot)=\left(
\begin{array}
[c]{ccc}%
l_{x_{1}}^{1}(\cdot) & \cdots & l_{x_{n}}^{1}(\cdot)\\
\vdots & \ddots & \vdots\\
l_{x_{1}}^{n}(\cdot) & \cdots & l_{x_{n}}^{n}(\cdot)
\end{array}
\right)  , & l_{\theta}^{i}(\cdot)=\left(
\begin{array}
[c]{ccc}%
l_{\theta_{11}}^{i}(\cdot) & \cdots & l_{\theta_{1d}}^{i}(\cdot)\\
\vdots & \ddots & \vdots\\
l_{\theta_{n1}}^{i}(\cdot) & \cdots & l_{\theta_{nd}}^{i}(\cdot)
\end{array}
\right)  .
\end{array}
\]
For simplicity, denote
\begin{equation}%
\begin{array}
[c]{lll}%
l_{x}(t)=l_{x}(t,\bar{X}_{t},\bar{\theta}_{t}), & l_{\theta}(t)=l_{\theta
}(t,\bar{X}_{t},\bar{\theta}_{t}), & g_{w}(t)=g_{w}(t,\bar{X}_{t},\bar{Y}%
_{t},\bar{Z}_{t},\bar{\theta}_{t}),
\end{array}
\label{notation-coefficient-optimal}%
\end{equation}
where $w=x$, $y$, $z$, $\theta$.

Using the convexity of $U$ and
taking an arbitrary $\xi \in\mathcal{U}_{ad}$, we know,
for $\varepsilon\in [0,1]$,
\[
\xi^{\varepsilon}:=\bar{\xi}+\varepsilon(\xi-\bar{\xi}) \in \mathcal{U}_{ad}.
\]
Let $(X_{{}}^{\varepsilon}(\cdot),Y_{{}}^{\varepsilon
}(\cdot),Z_{{}}^{\varepsilon}(\cdot),\theta^{\varepsilon}(\cdot))$ be the state trajectory of
(\ref{state-eq-new}) associated with $\xi^{\varepsilon}$. To derive the first-order necessary condition in terms of small $\varepsilon$, we
consider the following two BSDEs:
\begin{equation}
\left\{
\begin{array}
[c]{rl}%
d\hat{X}_{t}= & -\left[  l_{x}(t)\hat{X}_{t}+l_{\theta}(t)\hat{\theta}%
_{t}\right]  dt+\hat{\theta}_{t}dW_{t},\text{ \ }t\in\lbrack0,T],\\
\hat{X}_{T}= & \xi-\bar{\xi}%
\end{array}
\right.  \label{new-form-x1}%
\end{equation}
and
\begin{equation}
\left\{
\begin{array}
[c]{ll}%
d\hat{Y}_{t}= & -\left[  \left(  g_{x}(t)\right)  ^{\intercal}\hat{X}%
_{t}+g_{y}(t)\hat{Y}_{t}+\left(  g_{z}(t)\right)  ^{\intercal}\hat{Z}%
_{t}+\left\langle g_{\theta}(t),\hat{\theta}_{t}\right\rangle \right]  dt\\
& +\hat{Z}_{t}^{\intercal}dW_{t},\ \ t\in\lbrack0,T],\\
\hat{Y}_{T}= & \left(  \Phi_{x}(\bar{\xi})\right)  ^{\intercal}(\xi-\bar{\xi
}),
\end{array}
\right.  \label{new-form-y1}%
\end{equation}
where $l_{\theta}(t)\hat{\theta}_{t}:=\left( \langle l_{\theta}^{1}(t),\hat{\theta}_{t}\rangle,\ldots, \langle l_{\theta}^{n}(t),\hat{\theta}_{t}\rangle \right)^{\intercal}$, $t \in [0,T]$.

Under Assumptions \ref{assum-1} and \ref{assum-2}, note that (\ref{new-form-x1}) is a linear BSDE with bounded coefficients.
According to the existence and uniqueness result of solution of BSDE (\cite{Karoui-Peng-Quenez-1997}, Theorem 5.1), we obtain:

\begin{lemma}
\label{well-posedness-hat-X-theta}
Let Assumptions \ref{assum-1} and \ref{assum-2} hold. Then (\ref{new-form-x1}) admits a unique solution $\left(\hat{X}(\cdot),\hat{\theta}(\cdot)\right)\in$\\
$\mathcal{S}_{\mathcal{F}}^{4\bar{p}^{\ast}}([0,T];\mathbb{R}^{n})\times \mathcal{M}_{\mathcal{F}}^{2,4\bar{p}^{\ast}}([0,T];\mathbb{R}^{n\times d})$.
\end{lemma}

From the a priori estimate for BSDEs (\cite{Karoui-Peng-Quenez-1997}, Proposition 5.1), we can obtain the following estimates by using the method in \cite{Ji-Peng-2008} similarly.

\begin{lemma}
\label{est-forward-expansion-lem}
Let Assumptions \ref{assum-1} and \ref{assum-2} hold. Then, for any $\beta\in\left(  1,4\bar{p}^{\ast}\right]$, we have
\begin{equation}
\mathbb{E}\left[  \sup\limits_{t\in\lbrack0,T]}\left\vert X_{t}^{\varepsilon
}-\bar{X}_{t}\right\vert ^{\beta}+\left(  \int_{0}^{T}\left\vert \theta
_{t}^{\varepsilon}-\bar{\theta}_{t}\right\vert ^{2}dt\right)  ^{\frac{\beta
}{2}}\right]  =O\left(  \varepsilon^{\beta}\right)
,\label{est-forward-expansion-1}%
\end{equation}%
\begin{equation}
\mathbb{E}\left[  \sup\limits_{t\in\lbrack0,T]}\left\vert X_{t}^{\varepsilon
}-\bar{X}_{t}-\varepsilon\hat{X}_{t}\right\vert ^{\beta}+\left(  \int_{0}%
^{T}\left\vert \theta_{t}^{\varepsilon}-\bar{\theta}_{t}-\varepsilon
\hat{\theta}_{t}\right\vert ^{2}dt\right)  ^{\frac{\beta}{2}}\right]
=o\left(  \varepsilon^{\beta}\right)  .\label{est-forward-expansion-2}%
\end{equation}
\end{lemma}

For $t \in [0,T]$, we set
\[
\begin{array}
[c]{l}%
\left(  \chi_{t}^{1,\varepsilon},\eta_{t}^{1,\varepsilon},\zeta_{t}%
^{1,\varepsilon},\Theta_{t}^{1,\varepsilon}\right)  :=\left(  X_{t}%
^{\varepsilon}-\bar{X}_{t},Y_{t}^{\varepsilon}-\bar{Y}_{t},Z_{t}^{\varepsilon
}-\bar{Z}_{t},\theta_{t}^{\varepsilon}-\bar{\theta}_{t}\right)  ,\\
\text{ \ \ \ \ \ }\Lambda_{t}:=(\bar{X}_{t},\bar{Y}_{t},\bar{Z}_{t}%
,\bar{\theta}_{t}),\text{ \ \ \ }\Lambda_{t}^{\varepsilon}:=(X_{t}%
^{\varepsilon},Y_{t}^{\varepsilon},Z_{t}^{\varepsilon},\theta_{t}%
^{\varepsilon}).
\end{array}
\]
Then we have%
\begin{equation}
\left\{
\begin{array}
[c]{rl}%
d\eta_{t}^{1,\varepsilon}= & -\left[  \left(  \tilde{g}_{x}^{\varepsilon
}(t)\right)  ^{\intercal}\chi_{t}^{1,\varepsilon}+\tilde{g}_{y}^{\varepsilon
}(t)\eta_{t}^{1,\varepsilon}+\left(  \tilde{g}_{z}^{\varepsilon}(t)\right)
^{\intercal}\zeta_{t}^{1,\varepsilon}+\left\langle \tilde{g}_{\theta
}^{\varepsilon}(t),\Theta_{t}^{1,\varepsilon}\right\rangle \right]  dt\\
& +\left(  \zeta_{t}^{1,\varepsilon}\right)  ^{\intercal}dW_{t},\ \ t\in
\lbrack0,T],\\
\eta_{T}^{1,\varepsilon}= & \Phi(\xi^{\varepsilon})-\Phi(\bar{\xi}),
\end{array}
\right.  \label{ep-bar-y}%
\end{equation}
where $\tilde{g}_{z}^{\varepsilon}(t)=\int_{0}^{1}g_{z}(t,\Lambda_{t}%
+\nu(\Lambda_{t}^{\varepsilon}-\Lambda_{t}))d\nu$; $\tilde{g}%
_{x}^{\varepsilon}(t)$, $\tilde{g}_{y}^{\varepsilon}(t)$, $\tilde{g}%
_{\theta}^{\varepsilon}(t)$ are defined similarly.

\begin{remark}
\label{remark-C-uniform} Due to the second inequality in
(\ref{est-quad-qfbsde}) and $\left\vert \tilde{g}_{z}^{\varepsilon
}(t)\right\vert \leq L\left(1+  \left\vert \bar{Z}_{t}\right\vert
+\left\vert Z_{t}^{\varepsilon}\right\vert \right)  $, it can be verified that
$\sup_{\varepsilon\in\lbrack0,1]}\left\Vert \tilde{g}_{z}^{\varepsilon}\cdot
W\right\Vert _{\mathrm{BMO}_{2}}^{2}<3L^{2}(T+2A)$. Then, recalling (\ref{Rp-cd-func}), by definition of $p_{\tilde{g}_{z}^{\varepsilon}}$ and $\bar{p}$, we have
$\sup_{\varepsilon\in\lbrack0,1]}p_{\tilde{g}_{z}^{\varepsilon}}^{\ast}%
<\bar{p}^{\ast}$.
\end{remark}

The result below follows from Proposition \ref{est-linear-prop}, which provides the estimate for $\left(  \eta^{1,\varepsilon}%
(\cdot),\zeta^{1,\varepsilon}(\cdot)\right)  $.

\begin{lemma}
\label{est-epsilon-bar} Let Assumptions \ref{assum-1} and \ref{assum-2} hold. Then, for any $\beta\in\left(  1,2\bar{p}^{\ast
}\right)$, we have
\begin{equation}
\mathbb{E}\left[  \sup\limits_{t\in\lbrack0,T]}\left\vert \eta_{t}^{1,\varepsilon
}\right\vert ^{\beta}+\left(  \int_{0}^{T}\left\vert \zeta_{t}^{1,\varepsilon
}\right\vert ^{2}dt\right)  ^{\frac{\beta}{2}}\right]  =O\left(
\varepsilon^{\beta}\right).
\label{est-ep-yz}%
\end{equation}
\end{lemma}

\begin{proof}
Under Assumptions \ref{assum-1} and \ref{assum-2}, we get
\[%
\begin{array}
[c]{l}%
\left\vert \left(  \tilde{g}_{x}^{\varepsilon}(t)\right)  ^{\intercal}\chi
_{t}^{1,\varepsilon}\right\vert \leq C\left(  1+\left\vert Y_{t}^{\varepsilon
}\right\vert +\left\vert \bar{Y}_{t}\right\vert +\left\vert Z_{t}%
^{\varepsilon}\right\vert ^{2}+\left\vert \bar{Z}_{t}\right\vert ^{2}\right)
\left\vert \chi_{t}^{1,\varepsilon}\right\vert ,\\
\left\vert \langle\tilde{g}_{\theta}^{\varepsilon}(t),\Theta_{t}%
^{1,\varepsilon}\rangle\right\vert \leq C\left(  1+\left\vert Y_{t}%
^{\varepsilon}\right\vert +\left\vert \bar{Y}_{t}\right\vert +\left\vert
Z_{t}^{\varepsilon}\right\vert +\left\vert \bar{Z}_{t}\right\vert \right)
\left\vert \Theta_{t}^{1,\varepsilon}\right\vert .
\end{array}
\]
Due to (\ref{est-quad-qfbsde}), $Y^{\varepsilon}(\cdot)$, $\bar{Y}(\cdot)$ are both bounded by $A$.
Therefore, for any $\beta\in\left(  1,2\bar{p}^{\ast
}\right)  $, by Proposition \ref{est-linear-prop} and Remark
\ref{remark-C-uniform}, we obtain
\begin{equation}
\label{est-eta-zeta}
\begin{array}
[c]{l}%
\mathbb{E}\left[  \sup\limits_{t\in\lbrack0,T]}\left\vert \eta_{t}%
^{1,\varepsilon}\right\vert ^{\beta}+\left(  \int_{0}^{T}\left\vert \zeta
_{t}^{1,\varepsilon}\right\vert ^{2}dt\right)  ^{\frac{\beta}{2}}\right]  \\
\leq C\left(  \mathbb{E}\left[  \left\vert \Phi(\xi^{\varepsilon})-\Phi
(\bar{\xi})\right\vert ^{2\bar{p}^{\ast}}+\left(  \int_{0}^{T}\left\vert
\left(  \tilde{g}_{x}^{\varepsilon}(t)\right)  ^{\intercal}\chi_{t}%
^{1,\varepsilon}+\left\langle \tilde{g}_{\theta}^{\varepsilon}(t),\Theta
_{t}^{1,\varepsilon}\right\rangle \right\vert dt\right)  ^{2\bar{p}^{\ast}%
}\right]  \right)  ^{\frac{\beta}{2\bar{p}^{\ast}}}\\
\leq C\left\{  \left(  \mathbb{E}\left[  \left\vert (\xi-\bar{\xi})\right\vert
^{2\bar{p}^{\ast}}\right]  \right)  ^{\frac{\beta}{2\bar{p}^{\ast}}%
}\varepsilon^{\beta}+\left(  \mathbb{E}\left[  \left(  \int_{0}^{T}\left(
1+\left\vert Z_{t}^{\varepsilon}\right\vert ^{2}+\left\vert \bar{Z}%
_{t}\right\vert ^{2}\right)  \left\vert \chi_{t}^{1,\varepsilon}\right\vert
dt\right)  ^{2\bar{p}^{\ast}}\right]  \right)  ^{\frac{\beta}{2\bar{p}^{\ast}%
}}\right.  \\
\ \ +\left.  \left(  \mathbb{E}\left[  \left(  \int_{0}^{T}\left(
1+\left\vert Z_{t}^{\varepsilon}\right\vert +\left\vert \bar{Z}_{t}\right\vert
\right)  \left\vert \Theta_{t}^{1,\varepsilon}\right\vert dt\right)
^{2\bar{p}^{\ast}}\right]  \right)  ^{\frac{\beta}{2\bar{p}^{\ast}}}\right\},
\end{array}
\end{equation}
where the constant $C$ is independent of $\varepsilon$.
For the second term in the last inequality of (\ref{est-eta-zeta}),
by (\ref{est-forward-expansion-1}) and Corollary \ref{Zu-bounded},
it follows from H\"{o}lder's inequality that
\[
\begin{array}
[c]{rl}
& \mathbb{E}\left[  \left(  \int_{0}^{T}\left(  1+\left\vert Z_{t}%
^{\varepsilon}\right\vert ^{2}+\left\vert \bar{Z}_{t}\right\vert ^{2}\right)
\left\vert \chi_{t}^{1,\varepsilon}\right\vert dt\right)  ^{2\bar{p}^{\ast}%
}\right]  \\
\leq & \mathbb{E}\left[  \sup\limits_{t\in\lbrack0,T]}\left\vert \chi
_{t}^{1,\varepsilon}\right\vert ^{2\bar{p}^{\ast}}\left(  \int_{0}^{T}\left(
1+\left\vert Z_{t}^{\varepsilon}\right\vert ^{2}+\left\vert \bar{Z}%
_{t}\right\vert ^{2}\right)  dt\right)  ^{2\bar{p}^{\ast}}\right]  \\
\leq & \left(  \mathbb{E}\left[  \sup\limits_{t\in\lbrack0,T]}\left\vert
\chi_{t}^{1,\varepsilon}\right\vert ^{4\bar{p}^{\ast}}\right]  \right)
^{\frac{1}{2}}\left(  \mathbb{E}\left[  \left(  \int_{0}^{T}\left(
1+\left\vert Z_{t}^{\varepsilon}\right\vert ^{2}+\left\vert \bar{Z}%
_{t}\right\vert ^{2}\right)  dt\right)  ^{4\bar{p}^{\ast}}\right]  \right)
^{\frac{1}{2}}\\
\leq & C\varepsilon^{2\bar{p}^{\ast}};
\end{array}
\]
similarly, for the third term in the last inequality of (\ref{est-eta-zeta}), we get
\[%
\begin{array}
[c]{rl}
& \mathbb{E}\left[  \left(  \int_{0}^{T}\left(  1+\left\vert Z_{t}%
^{\varepsilon}\right\vert +\left\vert \bar{Z}_{t}\right\vert \right)
\left\vert \Theta_{t}^{1,\varepsilon}\right\vert dt\right)  ^{2\bar{p}^{\ast}%
}\right]  \\
\leq & 3^{\bar{p}^{\ast}}\mathbb{E}\left[  \left(  \int_{0}^{T}\left(  1+\left\vert
Z_{t}^{\varepsilon}\right\vert ^{2}+\left\vert \bar{Z}_{t}\right\vert
^{2}\right)  dt\right)  ^{\bar{p}^{\ast}}\left(  \int_{0}^{T}\left\vert
\Theta_{t}^{1,\varepsilon}\right\vert ^{2}dt\right)  ^{\bar{p}^{\ast}}\right]
\\
\leq & 3^{\bar{p}^{\ast}}\left(  \mathbb{E}\left[  \left(  \int_{0}^{T}\left(  1+\left\vert
Z_{t}^{\varepsilon}\right\vert ^{2}+\left\vert \bar{Z}_{t}\right\vert
^{2}\right)  dt\right)  ^{2\bar{p}^{\ast}}\right]  \right)  ^{\frac{1}{2}%
}\left(  \mathbb{E}\left[  \left(  \int_{0}^{T}\left\vert \Theta
_{t}^{1,\varepsilon}\right\vert ^{2}dt\right)  ^{2\bar{p}^{\ast}}\right]
\right)  ^{\frac{1}{2}}\\
\leq & C\varepsilon^{2\bar{p}^{\ast}}.
\end{array}
\]
Consequently, we have
\[
\mathbb{E}\left[  \sup\limits_{t\in\lbrack0,T]}\left\vert \eta_{t}^{1,\varepsilon
}\right\vert ^{\beta}+\left(  \int_{0}^{T}\left\vert \zeta_{t}^{1,\varepsilon
}\right\vert ^{2}dt\right)  ^{\frac{\beta}{2}}\right] \leq C\varepsilon^{\beta},
\]
where $C$ is independent of $\varepsilon$. The proof is complete.
\end{proof}

\begin{corollary}
\label{cor-1}
For any $p \in \left[ 1, \bar{p}^{\ast} \right) $, (\ref{est-ep-yz}) implies that
\[
\lim_{\varepsilon\rightarrow0^{+}}\mathbb{E}\left[  \left\vert \left(
\int_{0}^{T}\left\vert Z_{t}^{\varepsilon}\right\vert ^{2}dt\right)
^{p}-\left(  \int_{0}^{T}\left\vert \bar{Z}_{t}\right\vert ^{2}dt\right)
^{p}\right\vert \right]  =0.
\]
\end{corollary}

Now we prove the well-posedness of (\ref{new-form-y1}).

\begin{lemma}
\label{exist-y1-lem} Let Assumptions \ref{assum-1} and \ref{assum-2} hold. Then there
exists a unique solution $(\hat{Y}(\cdot),\hat{Z}(\cdot))\in\mathcal{S}%
_{\mathcal{F}}^{\beta}([0,T];\mathbb{R})\times\mathcal{M}_{\mathcal{F}%
}^{2,\beta}([0,T];\mathbb{R}^{d})$ to (\ref{new-form-y1}) for all
$\beta\in\left(  1,2\bar{p}^{\ast}\right)  $.
\end{lemma}

\begin{proof}
Under Assumption \ref{assum-1}, we get
\[%
\begin{array}
[c]{l}%
\left\vert \left(  g_{x}(t)\right)  ^{\intercal}\hat{X}_{t}\right\vert \leq
C\left(  1+\left\vert \bar{Y}_{t}\right\vert +\left\vert \bar{Z}%
_{t}\right\vert ^{2}\right)  \left\vert \hat{X}_{t}\right\vert ,\\
\left\vert \langle g_{\theta}(t),\hat{\theta}_{t}\rangle\right\vert \leq
C\left(  1+\left\vert \bar{Y}_{t}\right\vert +\left\vert \bar{Z}%
_{t}\right\vert \right)  \left\vert \hat{\theta}_{t}\right\vert .
\end{array}
\]
For any $\beta\in\left(  1,2\bar{p}^{\ast
}\right)  $, by Theorem \ref{state-eq-exist-th} and Lemma \ref{well-posedness-hat-X-theta},
one can check that (\ref{new-form-y1}) verifies the conditions in Proposition \ref{est-linear-prop}.
So it admits a unique solution $(\hat{Y}(\cdot),\hat{Z}(\cdot))\in \bigcap_{1<\beta<2\bar{p}^{\ast}}\left(\mathcal{S}
_{\mathcal{F}}^{\beta}([0,T];\mathbb{R})\times\mathcal{M}_{\mathcal{F}
}^{2,\beta}([0,T];\mathbb{R}^{d})\right)$ to (\ref{new-form-y1}).
Moreover, from (\ref{est-linear-data}), since $\bar{Y}(\cdot) \in L_{\mathcal{F}}^{\infty}([0,T];\mathbb{R}^{n})$, we have
\[%
\begin{array}
[c]{l}%
\mathbb{E}\left[  \sup\limits_{t\in\lbrack0,T]}\left\vert \hat{Y}%
_{t}\right\vert ^{\beta}+\left(  \int_{0}^{T}\left\vert \hat{Z}_{t}\right\vert
^{2}dt\right)  ^{\frac{\beta}{2}}\right]  \\
\leq C\left(  \mathbb{E}\left[  \left\vert \left(  \Phi_{x}(\bar{\xi})\right)
^{\intercal}(\xi-\bar{\xi})\right\vert ^{2\bar{p}^{\ast}}+\left(  \int_{0}%
^{T}\left\vert \left(  g_{x}(t)\right)  ^{\intercal}\hat{X}_{t}+\left\langle
g_{\theta}(t),\hat{\theta}_{t}\right\rangle \right\vert dt\right)  ^{2\bar
{p}^{\ast}}\right]  \right)  ^{\frac{\beta}{2\bar{p}^{\ast}}}\\
\leq C\left\{  \left(  \mathbb{E}\left[  \left\vert (\xi-\bar{\xi})\right\vert
^{2\bar{p}^{\ast}}\right]  \right)  ^{\frac{\beta}{2\bar{p}^{\ast}}}\right.
\\
\ \ +\left.  \left(  \mathbb{E}\left[  \left(  \int_{0}^{T}\left(
1+\left\vert \bar{Z}_{t}\right\vert ^{2}\right)  \left\vert \hat{X}%
_{t}\right\vert dt\right)  ^{2\bar{p}^{\ast}}+\left(  \int_{0}^{T}\left(
1+\left\vert \bar{Z}_{t}\right\vert \right)  \left\vert \hat{\theta}%
_{t}\right\vert dt\right)  ^{2\bar{p}^{\ast}}\right]  \right)  ^{\frac{\beta
}{2\bar{p}^{\ast}}}\right\}.
\end{array}
\]
Recall $\left(\hat{X}(\cdot),\hat{\theta}(\cdot)\right)\in
\mathcal{S}_{\mathcal{F}}^{4\bar{p}^{\ast}}([0,T];\mathbb{R}^{n})\times \mathcal{M}_{\mathcal{F}}^{2,4\bar{p}^{\ast}}([0,T];\mathbb{R}^{n\times d})$ from Lemma \ref{well-posedness-hat-X-theta}.
Then, similarly to the proof of the estimate (\ref{est-eta-zeta}),
we can obtain
\[
\mathbb{E}\left[  \sup\limits_{t\in\lbrack0,T]}\left\vert \hat{Y}%
_{t}\right\vert ^{\beta}+\left(  \int_{0}^{T}\left\vert \hat{Z}_{t}\right\vert
^{2}dt\right)  ^{\frac{\beta}{2}}\right]  <\infty
\]
for all $\beta\in\left(  1,2\bar{p}^{\ast}\right)  $, which accomplishes the proof.
\end{proof}

Now we give the main result of this subsection.

\begin{lemma}
\label{est-eta2-zeta2}
Let Assumptions \ref{assum-1} and \ref{assum-2} hold. Then, for any
$\beta\in\left(  1,\bar{p}^{\ast}\right)  $,%
\begin{equation}
\mathbb{E}\left[  \sup\limits_{t\in\lbrack0,T]}\left\vert Y_{t}^{\varepsilon
}-\bar{Y}_{t}-\varepsilon\hat{Y}_{t}\right\vert ^{\beta}+\left(  \int_{0}%
^{T}\left\vert Z_{t}^{\varepsilon}-\bar{Z}_{t}-\varepsilon\hat{Z}%
_{t}\right\vert ^{2}dt\right)  ^{\frac{\beta}{2}}\right]  =o\left(
\varepsilon^{\beta}\right)  .%
\end{equation}

\end{lemma}

\begin{proof}
We use the notation $\left(  \chi_{t}^{1,\varepsilon},\eta_{t}^{1,\varepsilon
},\zeta_{t}^{1,\varepsilon},\Theta_{t}^{1,\varepsilon}\right)  $, $\tilde
{g}_{x}^{\varepsilon}(t)$, $\tilde{g}_{y}^{\varepsilon}(t)$, $\tilde{g}%
_{z}^{\varepsilon}(t)$, $\tilde{g}_{\theta}^{\varepsilon}(t)$ introduced in
the proof of Lemma \ref{est-epsilon-bar}. Setting $\left(  \chi_{t}%
^{2,\varepsilon},\eta_{t}^{2,\varepsilon},\zeta_{t}^{2,\varepsilon},\Theta
_{t}^{2,\varepsilon}\right)  :=\left(  \chi_{t}^{1,\varepsilon}-\varepsilon
\hat{X}_{t},\eta_{t}^{1,\varepsilon}-\varepsilon\hat{Y}_{t},\zeta
_{t}^{1,\varepsilon}-\varepsilon\hat{Z}_{t},\Theta_{t}^{1,\varepsilon
}-\varepsilon\hat{\theta}_{t}\right)  $, we have%
\begin{equation}
\left\{
\begin{array}
[c]{rl}%
d\eta_{t}^{2,\varepsilon}= & -\left[  \left(  g_{x}(t)\right)  ^{\intercal
}\chi_{t}^{2,\varepsilon}+g_{y}(t)\eta_{t}^{2,\varepsilon}+\left(
g_{z}(t)\right)  ^{\intercal}\zeta_{t}^{2,\varepsilon}+\left\langle g_{\theta
}(t),\Theta_{t}^{2,\varepsilon}\right\rangle +R^{\varepsilon}(t)\right]  dt\\
& +\left(  \zeta_{t}^{2,\varepsilon}\right)  ^{\intercal}dW_{t},\ \ t\in
\lbrack0,T],\\
\eta_{T}^{2,\varepsilon}= & \left[  \int_{0}^{1}\left(  \Phi_{x}(\bar{\xi
}+\lambda(\xi^{\varepsilon}-\bar{\xi}))-\Phi_{x}(\bar{\xi})\right)
d\lambda\right]  ^{\intercal}(\xi^{\varepsilon}-\bar{\xi})
\end{array}
\right.  \label{eta-eps-1}%
\end{equation}
where
\[%
\begin{array}
[c]{rl}%
R^{\varepsilon}(t)= & \left[  \tilde{g}_{x}^{\varepsilon}(t)-g_{x}(t)\right]
^{\intercal}\chi_{t}^{1,\varepsilon}+\left[  \tilde{g}_{y}^{\varepsilon
}(t)-g_{y}(t)\right]  \eta_{t}^{1,\varepsilon}\\
& +\left[  \tilde{g}_{z}^{\varepsilon}(t)-g_{z}(t)\right]  ^{\intercal}%
\zeta_{t}^{1,\varepsilon}+\left\langle \tilde{g}_{\theta}^{\varepsilon
}(t)-g_{\theta}(t),\Theta_{t}^{1,\varepsilon}\right\rangle ,
\end{array}
\]
For any $\beta\in\left(  1,\bar{p}^{\ast}\right)  $ and $\beta_{0}\in\left(
\beta\vee p_{g_{z}}^{\ast},\bar{p}^{\ast}\right)  $, by Proposition
\ref{est-linear-prop}, we have
\begin{equation}%
\begin{array}
[c]{l}%
\mathbb{E}\left[  \sup\limits_{t\in\lbrack0,T]}\left\vert \eta_{t}%
^{2,\varepsilon}\right\vert ^{\beta}  +  \left(
%TCIMACRO{\dint _{0}^{T}}%
%BeginExpansion
{\displaystyle\int_{0}^{T}}
%EndExpansion
\left\vert \zeta_{t}^{2,\varepsilon}\right\vert ^{2}dt\right)  ^{\frac{\beta
}{2}}\right]  \\
\leq C\left(  \mathbb{E}\left[  \left\vert \eta_{T}^{2,\varepsilon}\right\vert
^{\beta_{0}}+\left(
%TCIMACRO{\dint _{0}^{T}}%
%BeginExpansion
{\displaystyle\int_{0}^{T}}
%EndExpansion
\left\vert \left(  g_{x}(t)\right)  ^{\intercal}\chi_{t}^{2,\varepsilon
}+\left\langle g_{\theta}(t),\Theta_{t}^{2,\varepsilon}\right\rangle
+R^{\varepsilon}(t)\right\vert dt\right)  ^{\beta_{0}}\right]  \right)
^{\frac{\beta}{\beta_{0}}}.
\end{array}
\label{est-eta2-zeta2-pre}%
\end{equation}
We only estimate the most difficult terms in (\ref{est-eta2-zeta2-pre}) as
follows. The other terms are similar.

Since $\left\vert g_{\theta}(t)\right\vert \leq L\left(  1+\left\vert
\bar{Y}_{t}\right\vert +\left\vert \bar{Z}_{t}\right\vert \right)  $ and
$\left\Vert \bar{Y}\right\Vert _{\infty}<\infty$, by Corollary
\ref{Zu-bounded} and (\ref{est-forward-expansion-2}), it follows from
H\"{o}lder's inequality that%
\[%
\begin{array}
[c]{l}%
\mathbb{E}\left[  \left(
%TCIMACRO{\dint _{0}^{T}}%
%BeginExpansion
{\displaystyle\int_{0}^{T}}
%EndExpansion
\left\vert \left\langle g_{\theta}(t),\Theta_{t}^{2,\varepsilon}\right\rangle\right\vert dt\right)  ^{\beta_{0}}\right]  \\
\leq C\mathbb{E}\left[  \left(
%TCIMACRO{\dint _{0}^{T}}%
%BeginExpansion
{\displaystyle\int_{0}^{T}}
%EndExpansion
\left(  1+\left\vert \bar{Z}_{t}\right\vert \right)  \left\vert \Theta
_{t}^{2,\varepsilon}\right\vert dt\right)  ^{\beta_{0}}\right]  \\
\leq C\left(  \mathbb{E}\left[  \left(
%TCIMACRO{\dint _{0}^{T}}%
%BeginExpansion
{\displaystyle\int_{0}^{T}}
%EndExpansion
\left\vert \Theta_{t}^{2,\varepsilon}\right\vert ^{2}dt\right)  ^{\beta_{0}%
}\right]  \right)  ^{\frac{1}{2}}\left(  \mathbb{E}\left[  \left(
%TCIMACRO{\dint _{0}^{T}}%
%BeginExpansion
{\displaystyle\int_{0}^{T}}
%EndExpansion
\left(  1+\left\vert \bar{Z}_{t}\right\vert ^{2}\right)  dt\right)
^{\beta_{0}}\right]  \right)  ^{\frac{1}{2}}\\
=o(\varepsilon^{\beta_{0}}).
\end{array}
\]

As $g_{z}$ is continuous with $(x,y,z,\theta)$, from
(\ref{est-forward-expansion-1}) and (\ref{est-ep-yz}), we have $\tilde{g}%
_{z}^{\varepsilon}(\cdot)$ converges to $g_{z}(\cdot)$ in the product measure
$\lambda\otimes\mathbb{P}$, where $\lambda$ denotes the Lebesgue measure on
$[0,T]$. Then, since $\left\vert \tilde{g}_{z}^{\varepsilon}(t)-g_{z}%
(t)\right\vert ^{2}\leq C\left(  1+\left\vert \bar{Z}_{t}\right\vert
^{2}+\left\vert Z_{t}^{\varepsilon}\right\vert ^{2}\right)  $, by Corollary
\ref{cor-1} and the generalized dominated convergence theorem (see Problem
16.6 (a) in \cite{Billingsley} or Problem 12 in \cite{Dudley}), we obtain%
\[
\lim_{\varepsilon\rightarrow 0}\mathbb{E}\left[  \int_{0}^{T}\left\vert
\tilde{g}_{z}^{\varepsilon}(t)-g_{z}(t)\right\vert ^{2}dt\right]  =0,
\]
which implies that $\left(  \int_{0}^{T}\left\vert \tilde{g}_{z}^{\varepsilon
}(t)-g_{z}(t)\right\vert ^{2}dt\right)  ^{\beta_{0}}\rightarrow0$ in
$\mathbb{P}$ as $\varepsilon\rightarrow0$. Hence, applying Corollary
\ref{cor-1} and the generalized dominated convergence theorem again, we get
\begin{equation}
\lim_{\varepsilon\rightarrow 0}\mathbb{E}\left[  \left(  \int_{0}%
^{T}\left\vert \tilde{g}_{z}^{\varepsilon}(t)-g_{z}(t)\right\vert
^{2}dt\right)  ^{\beta_{0}}\right]  =0.\label{est-fzepsilon-fzbar}%
\end{equation}
Consequently, due to (\ref{est-ep-yz}) and (\ref{est-fzepsilon-fzbar}), we obtain
\begin{equation}%
\begin{array}
[c]{l}%
\mathbb{E}\left[  \left(  \int_{0}^{T}\left\vert \tilde{g}_{z}^{\varepsilon
}(t)-g_{z}(t)\right\vert \left\vert \zeta^{1,\varepsilon}(t)\right\vert
dt\right)  ^{\beta_{0}}\right]  \\
\leq C\mathbb{E}\left[  \left(  \int_{0}^{T}\left\vert \tilde{g}%
_{z}^{\varepsilon}(t)-g_{z}(t)\right\vert ^{2}dt\right)  ^{\frac{\beta_{0}}%
{2}}\left(  \int_{0}^{T}\left\vert \zeta_{t}^{1,\varepsilon}\right\vert
^{2}dt\right)  ^{\frac{\beta_{0}}{2}}\right]  \\
\leq C\left(  \mathbb{E}\left[  \left(  \int_{0}^{T}\left\vert \tilde{g}%
_{z}^{\varepsilon}(t)-g_{z}(t)\right\vert ^{2}dt\right)  ^{\beta_{0}}\right]
\right)  ^{\frac{1}{2}}\left(  \mathbb{E}\left[  \left(  \int_{0}%
^{T}\left\vert \zeta_{t}^{1,\varepsilon}\right\vert ^{2}dt\right)  ^{\beta
_{0}}\right]  \right)  ^{\frac{1}{2}}\\
=o\left(  \varepsilon^{\beta_{0}}\right)  .
\end{array}
\label{est-exp-fzeps-fz-zeta1}%
\end{equation}
Consequently, we have
\[
\mathbb{E}\left[  \sup\limits_{t\in\lbrack0,T]}\left\vert \eta_{t}%
^{2,\varepsilon}\right\vert ^{\beta}  +  \left(
%TCIMACRO{\dint _{0}^{T}}%
%BeginExpansion
{\displaystyle\int_{0}^{T}}
%EndExpansion
\left\vert \zeta_{t}^{2,\varepsilon}\right\vert ^{2}dt\right)  ^{\frac{\beta
}{2}}\right] = o(\varepsilon^{\beta}),
\]
which accomplishes the proof.
\end{proof}

\subsection{Variational inequality}
In this subsection, we employ Ekeland's variational principle \cite{Ekeland-1974} to deal with the initial constraint $X_{0}^{\xi}=x_{0}$.

Given the optimal $\bar{\xi}$, we introduce a mapping $J_{\delta}: \mathcal{U}_{ad}\longmapsto \mathbb{R}$ by
\[
J_{\delta}(\xi):=\sqrt{\left\vert X_{0}^{\xi}-x_{0}\right\vert ^{2}+\left(
\max\left\{  0,J(\xi)-J(\bar{\xi})+\delta\right\}  \right)  ^{2}},
\]
where $x_{0}$ is the given initial state constraint and $\delta$ is an arbitrary positive constant. Let us check that the mappings $\xi \longmapsto \left \vert X_{0}^{\xi} - x_{0} \right \vert^{2}$, $\xi \longmapsto J(\xi)$, both from $\mathcal{U}_{ad}$ to $\mathbb{R}$, are continuous functional on $\mathcal{U}_{ad}$.

\begin{lemma}
\label{lem-Ekeland-cost-functional-continuity}
Let Assumptions \ref{assum-1} and \ref{assum-2} hold. Then $\left \vert X_{0}^{\xi} - x_{0} \right \vert^{2}$ and $J(\xi)$ are both continuous functional on $\mathcal{U}_{ad}$.
\end{lemma}

\begin{proof}
Under Assumptions \ref{assum-1} and \ref{assum-2}, since $Y_{0}^{\xi}$ is bounded from Theorem \ref{state-eq-exist-th} and $d(\xi_{m},\xi)\rightarrow 0$ implies that $\mathbb{E}[\left\vert \xi_{m} \right\vert^{4\bar{p}^{\ast}}]\rightarrow \mathbb{E}[\left\vert \xi\right\vert^{4\bar{p}^{\ast}} ]$
as $m \rightarrow \infty$ for any $\{\xi_{m}\}_{m\in\mathbb{N}_{+}}$, $\xi$ in $\mathcal{U}_{ad}$, we need only to show that $X_{0}^{\xi}$
and $Y_{0}^{\xi}$ are continuous on $\mathcal{U}_{ad}$. To do this, for
any given $\xi_{1}$, $\xi_{2}\in\mathcal{U}_{ad}$, let $\left(
X_{1}(\cdot),Y_{1}(\cdot),Z_{1}(\cdot),\theta_{1}(\cdot)\right)  $, $\left(  X_{2}(\cdot),Y_{2}%
(\cdot),Z_{2}(\cdot),\theta_{2}(\cdot)\right)  $ are respectively corresponding state
trajectories to $\xi_{1}$, $\xi_{2}$, satisfying (\ref{state-eq-new}). Then, by Proposition 5.1 in \cite{Karoui-Peng-Quenez-1997}, we have%
\begin{equation}
\mathbb{E}\left[  \sup\limits_{t\in\lbrack0,T]}\left\vert X_{1}(t)-X_{2}%
(t)\right\vert ^{4\bar{p}^{\ast}}+\left(  \int_{0}^{T}\left\vert \theta
_{1}(t)-\theta_{2}(t)\right\vert ^{2}dt\right)  ^{2\bar{p}^{\ast}}\right]
\leq C\mathbb{E}\left[  \left\vert \xi_{1}-\xi_{2}\right\vert ^{4\bar{p}%
^{\ast}}\right]  ,\label{ineq-X-xi-continuous}%
\end{equation}
which implies the continuity of $X_{0}^{\xi}$, where $C$ is a positive
constant depending on $T$, $L$, $\bar{p}^{\ast}$. On the other hand, for
$t\in\lbrack0,T]$, note that%
\[%
\begin{array}
[c]{rl}
& g(t,X_{1}(t),Y_{1}(t),Z_{1}(t),\theta_{1}(t))-g(t,X_{2}(t),Y_{2}%
(t),Z_{2}(t),\theta_{2}(t))\\
= & \lambda_{t}(Y_{1}(t)-Y_{2}(t))+\mu_{t}^{\intercal}(Z_{1}(t)-Z_{2}%
(t))+\varphi_{t},
\end{array}
\]
where
\[%
\begin{array}
[c]{l}%
\lambda_{t}=\frac{g(t,X_{1}(t),Y_{1}(t),Z_{1}(t),\theta_{1}(t))-g(t,X_{1}%
(t),Y_{2}(t),Z_{1}(t),\theta_{1}(t))}{Y_{1}(t)-Y_{2}(t)}\mathbf{1}%
_{\{Y_{1}\neq Y_{2}\}}(t),\\
\mu_{t}=\frac{g(t,X_{1}(t),Y_{2}(t),Z_{1}(t),\theta_{1}(t))-g(t,X_{1}%
(t),Y_{2}(t),Z_{2}(t),\theta_{1}(t))}{Z_{1}(t)-Z_{2}(t)}\mathbf{1}%
_{\{Z_{1}\neq Z_{2}\}}(t),\\
\varphi_{t}=g(t,X_{1}(t),Y_{2}(t),Z_{2}(t),\theta_{1}(t))-g(t,X_{2}%
(t),Y_{2}(t),Z_{2}(t),\theta_{2}(t)).
\end{array}
\]
Under Assumptions \ref{assum-1} and \ref{assum-2}, one can verify that $\lambda(\cdot)\in
L_{\mathcal{F}}^{\infty}([0,T];\mathbb{R})$, $\mu\cdot W\in\mathrm{BMO}$ with
$\left\Vert \mu\cdot W\right\Vert _{\mathrm{BMO}_{2}}^{2}<\Psi(\bar{p})$. Moreover,
by Theorem \ref{state-eq-exist-th}, Corollary \ref{Zu-bounded} and (\ref{ineq-X-xi-continuous}), we deduce $\varphi(\cdot)\in\mathcal{M}%
_{\mathcal{F}}^{1,2\bar{p}^{\ast}}([0,T];\mathbb{R})$. Then, by Proposition
\ref{est-linear-prop}, for any $\beta\in\left(  1,2\bar{p}^{\ast}\right)  $, we have
\[%
\begin{array}
[c]{rl}
& \mathbb{E}\left[  \sup\limits_{t\in\lbrack0,T]}\left\vert Y_{1}%
(t)-Y_{2}(t)\right\vert ^{\beta}+\left(  \int_{0}^{T}\left\vert Z_{1}%
(t)-Z_{2}(t)\right\vert ^{2}dt\right)  ^{\beta}\right]  \\
\leq & C\left(  \mathbb{E}\left[  \left\vert \Phi(\xi_{1})-\Phi(\xi
_{2})\right\vert ^{2\bar{p}^{\ast}}+\left(  \int_{0}^{T}\left\vert \varphi
_{t}\right\vert dt\right)  ^{2\bar{p}^{\ast}}\right]  \right)  ^{\frac{\beta
}{2\bar{p}^{\ast}}},
\end{array}
\]
where $C$ is a positive constant depending on $T$, $L$, $\left \Vert g_{y} \right \Vert_{\infty}$, $\bar{p}^{\ast}$,
$\beta$, $A$. Recall that $\Phi$ is Lipschitz continuous and observe that%
\[
\left\vert \varphi_{t}\right\vert \leq L\left[  \left(  1+\left\vert
Y_{2}(t)\right\vert +\left\vert Z_{2}(t)\right\vert ^{2}\right)  \left\vert
X_{1}(t)-X_{2}(t)\right\vert +\left(  1+\left\vert Y_{2}(t)\right\vert
+\left\vert Z_{2}(t)\right\vert \right)  \left\vert \theta_{1}(t)-\theta
_{2}(t)\right\vert \right]  .
\]
Using (\ref{ineq-X-xi-continuous}), similarly to the proof of (\ref{est-eta-zeta}), we obtain
\begin{equation}
\label{ineq-Y-xi-continuous}
\mathbb{E}\left[  \sup\limits_{t\in\lbrack0,T]}\left\vert Y_{1}(t)-Y_{2}%
(t)\right\vert ^{\beta}+\left(  \int_{0}^{T}\left\vert Z_{1}(t)-Z_{2}%
(t)\right\vert ^{2}dt\right)  ^{\beta}\right]  \leq C\left(  \mathbb{E}\left[
\left\vert \xi_{1}-\xi_{2}\right\vert ^{4\bar{p}^{\ast}}\right]  \right)
^{\frac{\beta}{4\bar{p}^{\ast}}},
\end{equation}
which implies the continuity of $Y_{0}^{\xi}$. The proof is complete.
\end{proof}

\begin{theorem}
\label{thm-variational-inequality}
Let Assumptions \ref{assum-1} and \ref{assum-2} hold, $\bar{\xi}$ be an optimal control to (\ref{state-eq-new})-(\ref{cost-functional}). Then there exist a real number $a_{0} \geq 0$ and $a_{1} \in \mathbb{R}^{n}$, with $\left \vert a_{0} \right \vert^{2} + \left \vert a_{1} \right \vert^{2} \neq 0$, such that the following variational inequality holds
\begin{equation}
\mathbb{E}\left[\left\langle
a_{1},\hat{X}_{0}\right\rangle +a_{0}\left\langle h_{x}(\bar{\xi},\bar{Y}_{0}),\xi-\bar{\xi
}\right\rangle +a_{0}h_{y}(\bar{\xi},\bar{Y}_{0})\hat{Y}_{0} \right]  \geq0,\label{variational-inequality}%
\end{equation}
where $\hat{X}_{0}$, $\hat{Y}_{0}$ are solutions to (\ref{new-form-x1}), (\ref{new-form-y1}) respectively.
\end{theorem}

\begin{proof}
Due to Lemma \ref{lem-Ekeland-cost-functional-continuity}, $J_{\delta}(\cdot)$ is continuous on $\mathcal{U}_{ad}[0,T]$. In addition, it is easy to check the following properties hold:
\[%
\begin{array}
[c]{l}%
J_{\delta}(\bar{\xi})=\delta;\\
J_{\delta}(\xi)>0,\text{ }\forall\xi\in\mathcal{U}_{ad};\\
J_{\delta}(\bar{\xi})\leq\inf\limits_{\xi\in\mathcal{U}_{ad}}J_{\delta
}(\xi)+\delta.
\end{array}
\]
Thus, from Ekeland's variational principle \cite{Ekeland-1974}, $\exists \ \xi_{\delta} \in \mathcal{U}_{ad}$ such that
\[%
\begin{array}
[c]{l}%
\text{(i) }J_{\delta}(\xi_{\delta}) \leq J_{\delta}(\bar{\xi});\\
\text{(ii) }d(\xi_{\delta},\bar{\xi}) \leq \sqrt{\delta};\\
\text{(iii) }J_{\delta}(\xi)+\sqrt{\delta}d(\xi,\xi_{\delta}) \geq J_{\delta}(\xi_{\delta}),\text{ }\forall \xi \in \mathcal{U}_{ad}.
\end{array}
\]

For any $\xi\in\mathcal{U}_{ad}$, set $\xi_{\delta}^{\varepsilon}%
=\xi_{\delta}+\varepsilon(\xi-\xi_{\delta})$, $\varepsilon\in\lbrack0,1]$. Let
$\left(  X_{\delta}^{\varepsilon}(\cdot),Y_{\delta}^{\varepsilon}%
(\cdot),Z_{\delta}^{\varepsilon}(\cdot),\theta_{\delta}^{\varepsilon}%
(\cdot)\right)  $ (resp. $\left(  X_{\delta}(\cdot),Y_{\delta}(\cdot
),Z_{\delta}(\cdot),\theta_{\delta}(\cdot)\right)  $) be the state trajectory
corresponding to $\xi_{\delta}^{\varepsilon}$ (resp. $\xi_{\delta}$), and
$\left(  \hat{X}_{\delta}(\cdot),\hat{\theta}_{\delta}(\cdot)\right)  $,
$\left(  \hat{Y}_{\delta}(\cdot),\hat{Z}_{\delta}(\cdot)\right)  $ be the
solutions to (\ref{new-form-x1}), (\ref{new-form-y1}) respectively in which $\bar{\xi}$ is substituted by
$\xi_{\delta}$.
From (iii) the above, we conclude%
\begin{equation}
J_{\delta}(\xi_{\delta}^{\varepsilon})-J_{\delta}(\xi_{\delta})+\sqrt{\delta
}d(\xi_{\delta}^{\varepsilon},\xi_{\delta})\geq0.\label{application-Ekeland}%
\end{equation}
On the other hand, similarly to (\ref{est-forward-expansion-2}) and (\ref{est-eta2-zeta2}), we have%
\[
\lim_{\varepsilon\rightarrow0}\frac{1}{\varepsilon}\mathbb{E}\left[
\sup\limits_{t\in\lbrack0,T]}\left\vert X_{\delta}^{\varepsilon}%
(\cdot)-X_{\delta}(\cdot)-\varepsilon\hat{X}_{\delta}(\cdot)\right\vert
\right]  =0,
\]%
\[
\lim_{\varepsilon\rightarrow0}\frac{1}{\varepsilon}\mathbb{E}\left[
\sup\limits_{t\in\lbrack0,T]}\left\vert Y_{\delta}^{\varepsilon}(t)-Y_{\delta
}(t)-\varepsilon\hat{Y}_{\delta}(t)\right\vert \right]  =0.
\]
Thus we obtain $X_{\delta}^{\varepsilon}(0)-X_{\delta}(0)=\varepsilon\hat
{X}_{\delta}(0)+o(\varepsilon)$ and $Y_{\delta}^{\varepsilon}(0)-Y_{\delta
}(0)=\varepsilon\hat{Y}_{\delta}(0)+o(\varepsilon)$, which leads to the
following expansions:%
\[
\left\vert X_{\delta}^{\varepsilon}(0)-x_{0}\right\vert ^{2}-\left\vert
X_{\delta}(0)-x_{0}\right\vert ^{2}=2\varepsilon\left\langle X_{\delta
}(0)-x_{0},\hat{X}_{\delta}(0)\right\rangle +o(\varepsilon),
\]%
\[%
\begin{array}
[c]{rl}
& \left(  J(\xi_{\delta}^{\varepsilon})-J(\bar{\xi})+\delta\right)
^{2}-\left(  J(\xi_{\delta})-J(\bar{\xi})+\delta\right)  ^{2}\\
= & \left(  \mathbb{E}\left[  h(\xi_{\delta}^{\varepsilon},Y_{\delta
}^{\varepsilon}(0))-h(\bar{\xi},\bar{Y}_{0})\right]  +\delta\right)
^{2}-\left(  \mathbb{E}\left[  h(\xi_{\delta},Y_{\delta}(0))-h(\bar{\xi}%
,\bar{Y}_{0})\right]  +\delta\right)  ^{2}\\
= & 2\varepsilon\left(  \mathbb{E}\left[  h(\xi_{\delta},Y_{\delta}%
(0))-h(\bar{\xi},\bar{Y}_{0})\right]  +\delta\right)  \\
& \cdot\mathbb{E}\left[  \left\langle h_{x}(\xi_{\delta},Y_{\delta}%
(0)),\xi-\xi_{\delta}\right\rangle +h_{y}(\xi_{\delta},Y_{\delta}(0))\hat
{Y}_{\delta}(0)\right]  +o(\varepsilon).
\end{array}
\]
Now we consider the following two cases:

\noindent \textbf{Case 1: }There exists $\varepsilon_{0}>0$ such that $J(\xi_{\delta}^{\varepsilon})-J(\bar{\xi})+\delta>0$ for all
$\varepsilon\in(0,\varepsilon_{0})$. In this case,%
\[%
\begin{array}
[c]{rl}
& \lim\limits_{\varepsilon\rightarrow0}\frac{J_{\delta}(\xi_{\delta
}^{\varepsilon})-J_{\delta}(\xi_{\delta})}{\varepsilon}\\
= & \lim\limits_{\varepsilon\rightarrow0}\frac{1}{J_{\delta}(\xi_{\delta
}^{\varepsilon})+J_{\delta}(\xi_{\delta})}\cdot\frac{J_{\delta}^{2}%
(\xi_{\delta}^{\varepsilon})-J_{\delta}^{2}(\xi_{\delta})}{\varepsilon}\\
= & \frac{1}{J_{\delta}(\xi_{\delta})}\left(  \left\langle X_{\delta}%
(0)-x_{0},\hat{X}_{\delta}(0)\right\rangle \right.  \\
& +\left.  \left(  J(\xi_{\delta})-J(\bar{\xi})+\delta\right)  \mathbb{E}%
\left[  \left\langle h_{x}(\xi_{\delta},Y_{\delta}(0)),\xi-\xi_{\delta
}\right\rangle +h_{y}(\xi_{\delta},Y_{\delta}(0))\hat{Y}_{\delta}(0)\right]
\right)  .
\end{array}
\]
Dividing (\ref{application-Ekeland}) by $\varepsilon$ and sending
$\varepsilon$ to $0$, we obtain%
\begin{equation}%
\begin{array}
[c]{rl}
& \left\langle a_{1}^{\delta},\hat{X}_{\delta}(0)\right\rangle +a_{0}^{\delta
}\mathbb{E}\left[  \left\langle h_{x}(\xi_{\delta},Y_{\delta}(0)),\xi
-\xi_{\delta}\right\rangle +h_{y}(\xi_{\delta},Y_{\delta}(0))\hat{Y}_{\delta
}(0)\right]  \\
\geq & -\sqrt{\delta}\left(  \mathbb{E}\left[  \left\vert \xi-\xi_{\delta
}\right\vert ^{4\bar{p}^{\ast}}\right]  \right)  ^{\frac{1}{4\bar{p}^{\ast}}},
\end{array}
\label{approxi-variational-ineq}%
\end{equation}
where%
\[
a_{1}^{\delta}=\frac{1}{J_{\delta}(\xi_{\delta})}\left(  X_{0}^{\xi_{\delta}%
}-x_{0}\right)  ,\text{ \ }a_{0}^{\delta}=\frac{1}{J_{\delta}(\xi_{\delta}%
)}\left(  J(\xi_{\delta})-J(\bar{\xi})+\delta\right)  .
\]

\noindent \textbf{Case 2}: There exists a positive sequence $\left\{  \varepsilon
_{n}\right\}  $ satisfying $\varepsilon_{n}\rightarrow0$, such that
$J(\xi_{\delta}^{\varepsilon
_{n}})-J(\bar{\xi})+\delta\leq0$. In this case, by its definition $J_{\delta}(\xi_{\delta
}^{\varepsilon_{n}})=\sqrt{\left\vert X_{0}^{\xi_{\delta}^{\varepsilon_{n}}%
}-x_{0}\right\vert ^{2}}$ for sufficiently large $n$. Since $J_{\delta}%
(\cdot)$ is continuous on $\mathcal{U}_{ad}[0,T]$, we conclude $J_{\delta}%
(\xi_{\delta})=\sqrt{\left\vert X_{0}^{\xi_{\delta}}-x_{0}\right\vert ^{2}}$.
Now we have%
\[
\lim\limits_{n\rightarrow\infty}\frac{J_{\delta}(\xi_{\delta}^{\varepsilon
_{n}})-J_{\delta}(\xi_{\delta})}{\varepsilon_{n}}=\lim\limits_{n\rightarrow
\infty}\frac{1}{J_{\delta}(\xi_{\delta}^{\varepsilon_{n}})+J_{\delta}%
(\xi_{\delta})}\cdot\frac{J_{\delta}^{2}(\xi_{\delta}^{\varepsilon_{n}%
})-J_{\delta}^{2}(\xi_{\delta})}{\varepsilon_{n}}=\frac{\left\langle
X_{0}^{\xi_{\delta}}-x_{0},\hat{X}_{\delta}(0)\right\rangle }{J_{\delta}%
(\xi_{\delta})}.
\]
Similar to Case 1, we derive $\left\langle a_{1}^{\delta},\hat{X}_{\delta
}(0)\right\rangle \geq-\sqrt{\delta}\left(  \mathbb{E}\left[  \left\vert
\xi-\xi_{\delta}\right\vert ^{4\bar{p}^{\ast}}\right]  \right)  ^{\frac
{1}{4\bar{p}^{\ast}}}$, where $a_{0}^{\delta}=0$, $a_{1}^{\delta}=\frac
{X_{0}^{\xi_{\delta}}-x_{0}}{J_{\delta}(\xi_{\delta})}$.

In summary, for both cases, we have $a_{0}^{\delta}\geq0$, $\left\vert
a_{0}^{\delta}\right\vert ^{2}+\left\vert a_{1}^{\delta}\right\vert ^{2}=1$
and (\ref{approxi-variational-ineq}). Hence, there exist a convergent
subsequence of $(a_{0}^{\delta},a_{1}^{\delta})$ whose limit is denoted by
$(a_{0},a_{1})$. Due to $d(\xi_{\delta},\bar{\xi})\leq\sqrt{\delta}$, we have
$\xi_{\delta}\rightarrow\bar{\xi}$ in $\mathcal{U}_{ad}$, as
$\delta\rightarrow0$. So, in (\ref{ineq-X-xi-continuous}) and (\ref{ineq-Y-xi-continuous}), substituting $\xi_{\delta}$, $\bar{\xi}$ for $\xi_{1}$, $\xi_{2}$ respectively, we deduce that
\begin{equation}
\lim_{\delta\rightarrow0}\mathbb{E}\left[  \sup\limits_{t\in\lbrack
0,T]}\left\vert X_{\delta}(t)-\bar{X}_{t}\right\vert ^{4\bar{p}^{\ast}%
}+\left(  \int_{0}^{T}\left\vert \theta_{\delta}(t)-\bar{\theta}%
_{t}\right\vert ^{2}dt\right)  ^{2\bar{p}^{\ast}}\right]  =0,
\label{differ-X-xi-delta-bar}%
\end{equation}
\begin{equation}
\lim_{\delta\rightarrow0}\mathbb{E}\left[  \sup\limits_{t\in\lbrack
0,T]}\left\vert Y_{\delta}(t)-\bar{Y}_{t}\right\vert ^{\beta}+\left(  \int%
_{0}^{T}\left\vert Z_{\delta}(t)-\bar{Z}_{t}\right\vert ^{2}dt\right)
^{\beta}\right]  =0,
\label{differ-Y-xi-delta-bar}%
\end{equation}
where $\beta\in(1,2\bar{p}^{\ast})$ are arbitrarily given. Then, from (\ref{differ-X-xi-delta-bar}) and (\ref{differ-Y-xi-delta-bar}), for $\beta \in \left( 1, \bar{p}^{\ast}\right)$, one can prove
\[
\lim_{\delta\rightarrow0}\mathbb{E}\left[  \sup\limits_{t\in\lbrack
0,T]}\left\vert \hat{Y}_{\delta}(t)-\hat{Y}_{t}\right\vert ^{\beta}+\left(
\int_{0}^{T}\left\vert \hat{Z}_{\delta}(t)-\hat{Z}_{t}\right\vert
^{2}dt\right)  ^{\beta}\right]  =0
\]
similarly to the proof of Lemma \ref{est-eta2-zeta2}, which implies $\hat{Y}_{\delta}(0)\rightarrow\hat{Y}_{0}$.
$\hat{X}_{\delta}(0)\rightarrow\hat{X}_{0}$ can be deduced by applying Proposition 5.1 in \cite{Karoui-Peng-Quenez-1997}.
All in all, let $\delta \rightarrow 0$ in (\ref{approxi-variational-ineq}), we get (\ref{variational-inequality}). The proof is complete.
\end{proof}

\subsection{Maximum principle}
In this subsection, we derive the stochastic maximum principle. To this end, we introduce the adjoint process $\left(
p(\cdot),q(\cdot)\right)  $ associated with the optimal admissible control
$\bar{\xi}$ to (\ref{state-eq-new})-(\ref{cost-functional}), which solves the following adjoint system:%
\begin{equation}
\left\{
\begin{array}
[c]{rl}%
dp_{t}= & \left(  l_{x}(t)p_{t}+g_{x}(t)q_{t}\right)  dt+\left(  l_{\theta
}(t)p_{t}+g_{\theta}(t)q_{t}\right)  dW_{t},\\
p_{0}= & a_{1},\\
dq_{t}= & g_{y}(t)q_{t}dt+\left(  g_{z}(t)\right)  ^{\intercal}q_{t}dW_{t},\\
q_{0}= & a_{0}h_{y}(\bar{\xi},\bar{Y}_{0}),\text{ \ }t\in\lbrack0,T],
\end{array}
\right.  \label{adj-eq}%
\end{equation}
where $l_{x}(t)$, $\left\{  l_{\theta}^{i}(t)\right\}  _{i=1,\ldots n}$,
$g_{x}(t)$, $g_{y}(t)$, $g_{z}(t)$ are defined by (\ref{notation-coefficient-optimal}),
$l_{\theta}(t)p_{t}:=\left(  \left(  l_{\theta}^{1}(t)\right)  ^{\intercal
}p_{t},\ldots,\left(  l_{\theta}^{n}(t)\right)  ^{\intercal}p_{t}\right)
^{\intercal}$, and $a_{0}$, $a_{1}$ are as in Theorem \ref{thm-variational-inequality}.

Now we prove the well-posedness of (\ref{adj-eq}).
\begin{lemma}
\label{adj-1st-lem}
Let Assumptions \ref{assum-1} and \ref{assum-2} hold.
Then (\ref{adj-eq}) admits a unique strong solution $\left( p(\cdot), q(\cdot) \right)$. Moreover, for any given $\beta \in (1,\bar{p})$, we have
\[
\mathbb{E}\left[  \sup_{t\in\lbrack0,T]}\left\vert p_{t}\right\vert ^{\beta
}+\sup_{t\in\lbrack0,T]}\left\vert q_{t}\right\vert ^{\beta}\right]  <\infty.
\]
\end{lemma}

\begin{proof}
In (\ref{adj-eq}), we first consider the SDE where $q(\cdot)$ satisfies because the coefficients involved in it are no longer bounded. Under Assumptions \ref{assum-1} and \ref{assum-2}, one can check the coefficients satisfy the conditions in the basic theorem in [\cite{Gal'Chuk1978}, pp. 756-757] to the underlying semi-martingale
$((%
%TCIMACRO{\TeXButton{d number 1}{\overbrace{1,1,\cdot\cdot\cdot,1}^{d}}}%
%BeginExpansion
\overbrace{1,\ldots,1}^{d}%
%EndExpansion
)_{{}}^{\intercal}t+W_{t})_{t\in\lbrack0,T]}^{{}}$ (see also Lemma 7.1 in \cite{Tang03}). So it admits a unique strong solution $q(\cdot)$ up to an evanescent set. Moreover, set
\[
\tilde{q}_{t}= a_{0}h_{y}(\bar{\xi},\bar{Y}_{0})\exp\left\{\int_{0}^{t} g_{y}(s)ds\right\}\mathcal{E}\left(\int_{0}^{t}g_{z}(s)dW_{s}\right),\ \ t \in [0,T].
\]
Then, noting $\tilde{q}(\cdot)$ is continuous and applying It\^{o}'s lemma to $\tilde{q}_{t}$ on $[0,T]$, one can verify that $\tilde{q}(\cdot)=q(\cdot)$, $\mathbb{P}$-a.s.. On the other hand, for any $\beta \in (1,\bar{p})$, recalling Remark \ref{remark-C-uniform} and using reverse H\"{o}lder's inequality, we obtain
$\mathbb{E}\left[
\sup_{t\in\lbrack0,T]}\left\vert \tilde{q}_{t}\right\vert ^{\beta}\right]
\leq C$, where $C>0$ depends on $L$, $T$, $\left\Vert g_{y}\right\Vert
_{\infty}$, $A$, $a_{0}$, $\beta$.

Now let us focus on the SDE which $p(\cdot)$ satisfies. Under Assumptions \ref{assum-1} and \ref{assum-2}, since $l_{x}$ and $l_{\theta}$ are bounded, it admits a unique strong solution $p(\cdot)$ and, by using a standard estimate for SDEs, we get
\[
\mathbb{E}\left[  \sup\limits_{t\in\lbrack0,T]}\left\vert p_{t}\right\vert
^{\beta}\right]  \leq C\left\{  \left\vert a_{1}\right\vert ^{\beta
}+\mathbb{E}\left[  \left(  \int_{0}^{T}\left\vert g_{x}(t)q_{t}\right\vert
dt\right)  ^{\beta}+\left(  \int_{0}^{T}\left\vert g_{\theta}(t)q_{t}%
\right\vert ^{2}dt\right)  ^{\frac{\beta}{2}}\right]  \right\}  .
\]
The right-hand side of the above inequality is finite since we can show that, for any given $\beta \in (1,\bar{p})$ and any $\beta_{0} \in (\beta, \bar{p})$, by Corollary \ref{Zu-bounded} and using H\"{o}lder's inequality,
\[%
\begin{array}
[c]{rl}
& \mathbb{E}\left[  \left(  \int_{0}^{T}\left\vert g_{\theta}(t)q_{t}%
\right\vert ^{2}dt\right)  ^{\frac{\beta}{2}}\right]  \\
\leq & L^{\beta}\mathbb{E}\left[  \sup\limits_{t\in\lbrack0,T]}\left\vert
q_{t}\right\vert ^{\beta}\left(  \int_{0}^{T}\left(  1+\left\vert \bar{Y}%
_{t}\right\vert +\left\vert \bar{Z}_{t}\right\vert \right)  ^{2}dt\right)
^{\frac{\beta}{2}}\right]  \\
\leq & L^{\beta}\left(  \mathbb{E}\left[  \sup\limits_{t\in\lbrack
0,T]}\left\vert q_{t}\right\vert ^{\beta_{0}}\right]  \right)  ^{\frac{\beta
}{\beta_{0}}}\left(  \mathbb{E}\left[  \left(  \int_{0}^{T}\left(
1+\left\vert \bar{Z}_{t}\right\vert \right)  ^{2}dt\right)  ^{\frac{\beta
\beta_{0}}{2(\beta_{0}-\beta)}}\right]  \right)  ^{\frac{\beta_{0}-\beta
}{\beta_{0}}}\\
\leq & C,
\end{array}
\]
where $C>0$ depends on $L$, $T$, $\left\Vert g_{y}\right\Vert
_{\infty}$, $A$, $a_{0}$, $\beta$, $\beta_{0}$. The other term in that inequality can be estimated similarly. The proof is complete.
\end{proof}

\begin{theorem}
\label{thm-SMP}
Let Assumptions \ref{assum-1} and \ref{assum-2} hold. If $\bar{\xi}$ is optimal to (\ref{state-eq-new})-(\ref{cost-functional}), then there exist $a_{1} \in \mathbb{R}^{n}$ and $a_{0} \in \mathbb{R}$ with $a_{0}\geq 0$, $\left \vert a_{0} \right \vert^{2} + \left \vert a_{1} \right \vert^{2} \neq 0$ such that
\begin{equation}
\left\langle p_{T}+q_{T}\Phi_{x}(\bar{\xi})+a_{0}h_{x}(\bar{\xi},\bar{Y}%
_{0}),v-\bar{\xi}\right\rangle \geq0,\text{ \ }\forall v\in K,\text{
\ }\mathbb{P}\text{-a.s.},\label{ineq-SMP}%
\end{equation}
where $\left(  p(\cdot),q(\cdot)\right)  $ uniquely solves (\ref{adj-eq}).
\end{theorem}

\begin{proof}
For $\xi\in\mathcal{U}_{ad}$, let $\left(  \hat{X}(\cdot),\hat{\theta
}(\cdot)\right)  $, $\left(  \hat{Y}(\cdot),\hat{Z}(\cdot)\right)  $ be the
solution to (\ref{new-form-x1}), (\ref{new-form-y1}) respectively. Applying It\^{o}'s lemma to $\left\langle
p_{t},\hat{X}_{t}\right\rangle +q_{t}\hat{Y}_{t}$ on $[0,T]$ yields%
\begin{equation}
d\left[  \left\langle p_{t},\hat{X}_{t}\right\rangle +\left\langle q_{t}%
,\hat{Y}_{t}\right\rangle \right]  =\Gamma_{t}dW_{t},\label{Ito-lemma-applied}%
\end{equation}
where%
\begin{equation}
\Gamma_{t}=p_{t}^{\intercal}\hat{\theta}_{t}+q_{t}\left(  \hat{Z}_{t}%
+g_{z}(t)\hat{Y}_{t}\right)  +\hat{X}_{t}^{\intercal}\left[  l_{\theta
}(t)p_{t}+g_{\theta}(t)q_{t}\right]  ,\text{ }t\in\lbrack
0,T].\label{def-Gamma}%
\end{equation}
We claim that $\mathbb{E}\left[  \left(  \int_{0}^{T}\left\vert \Gamma
_{t}\right\vert ^{2}dt\right)  ^{\frac{1}{2}}\right]  <\infty$. Set $\beta
_{1}:=3\bar{p}(\bar{p}+2)^{-1}$, $\beta_{1}^{\ast}:=\beta_{1}(\beta
_{1}-1)^{-1}$, $\beta_{2}:=\frac{7}{4}\bar{p}(\bar{p}-1)^{-1}$, where $\bar
{p}$ is defined by (\ref{p-bar}). It is easy to check that $1<\beta_{1}<\bar{p}$ and
$\bar{p}^{\ast}<\beta_{1}^{\ast}<\beta_{2}<2\bar{p}^{\ast}$. Actually, in (\ref{def-Gamma}),
we can show that%
\[%
\begin{array}
[c]{rl}
& \mathbb{E}\left[  \left(  \int_{0}^{T}\left\vert q_{t}g_{z}(t)\hat{Y}%
_{t}\right\vert ^{2}dt\right)  ^{\frac{1}{2}}\right]  \\
\leq & 2L\mathbb{E}\left[  \sup\limits_{t\in\lbrack0,T]}\left\vert
q_{t}\right\vert \left(  \int_{0}^{T}\left(  1+\left\vert \bar{Z}%
_{t}\right\vert ^{2}\right)  \left\vert \hat{Y}_{t}\right\vert ^{2}dt\right)
^{\frac{1}{2}}\right]  \\
\leq & 2L\left(  \mathbb{E}\left[  \sup\limits_{t\in\lbrack0,T]}\left\vert
q_{t}\right\vert ^{\beta_{1}}\right]  \right)  ^{\frac{1}{\beta_{1}}}\left(
\mathbb{E}\left[  \sup\limits_{t\in\lbrack0,T]}\left\vert \hat{Y}%
_{t}\right\vert ^{\beta_{1}^{\ast}}\left(  \int_{0}^{T}\left(  1+\left\vert
\bar{Z}_{t}\right\vert ^{2}\right)  dt\right)  ^{\frac{\beta_{1}^{\ast}}{2}%
}\right]  \right)  ^{\frac{1}{\beta_{1}^{\ast}}}\\
\leq & 2L\left(  \mathbb{E}\left[  \sup\limits_{t\in\lbrack0,T]}\left\vert
q_{t}\right\vert ^{\beta_{1}}\right]  \right)  ^{\frac{1}{\beta_{1}}}\left(
\mathbb{E}\left[  \sup\limits_{t\in\lbrack0,T]}\left\vert \hat{Y}%
_{t}\right\vert ^{\beta_{2}}\right]  \right)  ^{\frac{1}{\beta_{2}}}\left(
\mathbb{E}\left[  \left(  \int_{0}^{T}\left(  1+\left\vert \bar{Z}%
_{t}\right\vert ^{2}\right)  dt\right)  ^{\frac{\beta_{1}^{\ast}\beta_{2}%
}{2(\beta_{2}-\beta_{1}^{\ast})}}\right]  \right)  ^{\frac{\beta_{2}-\beta
_{1}^{\ast}}{\beta_{1}^{\ast}\beta_{2}}}%
\end{array}
\]
by using H\"{o}lder's inequality. Then it follows from reverse H\"{o}lder's
inequality, Corollary \ref{Zu-bounded} and Lemma \ref{adj-1st-lem} that $\mathbb{E}\left[  \left(  \int%
_{0}^{T}\left\vert q_{t}g_{z}(t)\hat{Y}_{t}\right\vert ^{2}dt\right)
^{\frac{1}{2}}\right]  <\infty$. The other terms in (\ref{def-Gamma}) can be estimated similarly, so
$\int_{0}^{\cdot}\Gamma_{t}dW_{t}$ is a true martingale on $[0,T]$.

Integrating (\ref{Ito-lemma-applied}) from $0$ to $T$, taking expectation and using the
variational inequality (\ref{variational-inequality}), we obtain $\mathbb{E}\left[  \int_{0}^{T}\Gamma
_{t}dW_{t}\right]  =0$ and
\[%
\begin{array}
[c]{rl}
& \mathbb{E}\left[  \left\langle p_{T}+q_{T}\Phi_{x}(\bar{\xi})+a_{0}%
h_{x}(\bar{\xi},\bar{Y}_{0}),\xi-\bar{\xi}\right\rangle \right]  \\
= & \mathbb{E}\left[  \left\langle a_{1},\hat{X}_{0}\right\rangle +a_{0}%
h_{y}(\bar{\xi},\bar{Y}_{0})\hat{Y}_{0}+a_{0}\left\langle h_{x}(\bar{\xi}%
,\bar{Y}_{0}),\xi-\bar{\xi}\right\rangle \right]  \\
\geq & 0.
\end{array}
\]
Since $\xi\in\mathcal{U}_{ad}$ is arbitrary, a standard argument yields
(\ref{ineq-SMP}). The proof is complete.
\end{proof}

Denote $\partial K$ by the boundary of $K$. Set
$
\Omega_{0}:=\left\{  \omega\in\Omega\mid\bar{\xi}(\omega)\in\partial
K\right\}
$.
According to Theorem \ref{thm-SMP}, the following corollary holds.
\begin{corollary}
\label{cor-SMP}
Under assumptions of Theorem \ref{thm-SMP}, for each $v\in K$,
\[%
\begin{array}
[c]{l}%
\left\langle p_{T}+q_{T}\Phi_{x}(\bar{\xi})+a_{0}h_{x}(\bar{\xi},\bar{Y}%
_{0}),v-\bar{\xi}\right\rangle \geq0, \ \mathbb{P}\text{-a.s. on }\Omega_{0},\\
p_{T}+q_{T}\Phi_{x}(\bar{\xi})+a_{0}h_{x}(\bar{\xi},\bar{Y}_{0})=0, \ \mathbb{P}\text{-a.s. on }\Omega\setminus\Omega_{0}.
\end{array}
\]
\end{corollary}

\section{An application to a robust recursive utility maximization problem with bankruptcy prohibition}
There are $d+1$ investment instruments in the market. One of the instruments
is a bank account (free risk); the others are stocks. The price processes are
described by the following equations:%
\[
\left\{
\begin{array}
[c]{rl}%
dP_{t}^{0}= & P_{t}^{0}r_{t}dt,\\
P_{0}^{0}= & \kappa_{0}>0,\\
dP_{t}^{i}= & P_{t}^{i}\left[  b_{t}^{i}dt+\sum\limits_{j=1}^{d}\sigma
_{t}^{ij}dW_{t}^{j}\right]  ,\\
P_{0}^{i}= & \kappa_{i} >0,\text{ }i=1,\ldots,d,\text{ \ }t\in\lbrack0,T].
\end{array}
\right.
\]
where the interest rate $r(\cdot)$, the stock-appreciation rate $b(\cdot
):=\left(  b^{1}(\cdot),\ldots,b^{d}(\cdot)\right)  ^{\intercal}$, the the
stock-volatility $\sigma(\cdot):=\left\{  \sigma^{ij}(\cdot)\right\}  _{1\leq
i,j\leq d}\ $are all deterministic, bounded processes in suitable sizes.
Moreover, $r(\cdot)$ is assumed to be nonnegative and $\sigma(\cdot)$ is
assumed to be invertible whose inverse $\sigma^{-1}(\cdot)$ is also bounded.

An investor whose initial wealth is taken $x_{0} \geq 0$ as a primitive, decides to invest in the $i$th
stock $(i=1,\ldots,d)$ with the amount $\pi^{i}(\cdot)$. Denote $X(\cdot)$ and
$Y(\cdot)$ by the wealth process and the recursive utility of the investor,
respectively. Let $B_{t}:=\left(  b_{t}^{1}-r_{t},\ldots,b_{t}^{d}%
-r_{t}\right)  $; $\pi(\cdot)=\left(  \pi^{1}(\cdot),\ldots,\pi^{d}%
(\cdot)\right)  ^{\intercal}$ be the portfolio process and $\phi(\cdot
)=\sigma^{-1}(\cdot)B^{\intercal}(\cdot)$ be the risk premium process. Here, we suppose 
that the instantaneous consumption rate $c(\cdot)$ depends only on the wealth process $X(\cdot)$. Thus, by the conventional calculation, the wealth process $X(\cdot)$ satisfies the
following SDE:%
\begin{equation}
\left\{
\begin{array}
[c]{rl}%
dX_{t}^{\pi}= & \left[\left(  r_{t}X_{t}^{\pi}+\pi_{t}^{\intercal}\sigma_{t}\phi
_{t}\right) - c(X_{t}^{\pi}) \right] dt+\pi_{t}^{\intercal}\sigma_{t}dW_{t},\text{ \ }t\in
\lbrack0,T],\\
X_{0}^{\pi}= & x_{0},
\end{array}
\right.  \label{wealth-eq}%
\end{equation}
where the consumption function $c$ is nonnegative and continuous differentiable.
The recursive utility of the investor's wealth $X^{\pi}(\cdot)$ is
described by the following BSDE:
\begin{equation}
\left\{
\begin{array}
[c]{rl}%
dY_{t}^{\pi}= & -g(t,X_{t}^{\pi},Y_{t}^{\pi},Z_{t}^{\pi})dt+\left( Z_{t}^{\pi} \right)^{\intercal} dW_{t},\text{
\ }t\in\lbrack0,T],\\
Y_{T}^{\pi}= & \Phi(X_{T}^{\pi}),
\end{array}
\right.  \label{recursive-eq}%
\end{equation}
where $g$ and $\Phi$ satisfy Assumption \ref{assum-1}.

Our problem is that an investor chooses portfolio $\pi(\cdot)$ so as to maximize the recursive utility $Y_{0}^{\pi}$ of his wealth $X^{\pi}(\cdot)$ 
with bankruptcy prohibition. Equivalently, we put $h(x,y)=-y$ since the control problem section 3 is to minimize the cost functional, that is,
\[
\begin{array}
[c]{rl}%
\text{minimize} & \text{ }J(\pi(\cdot))=-Y_{0}^{\pi}\\
\text{subject to} & \text{ }\pi(\cdot)\in\mathcal{M}_{\mathcal{F}}^{2,4\bar
{p}^{\ast}}([0,T];\mathbb{R}^{d}),\text{ \ }X_{t}^{\pi}\geq0, \  t \in [0,T],
\end{array}
\]
where $\bar{p}^{\ast}$ is the exponential conjugate of $\bar{p}$ introduced by (\ref{p-bar}). 
Using the method in section 3, let $\theta_{t}=\sigma_{t}^{\intercal}\pi
_{t}$, then we get the following equivalent control system
\[
\left\{
\begin{array}
[c]{rl}%
dX_{t}^{\xi}= & -l(t,X_{t}^{\xi},\theta_{t}^{\xi})dt+\left(\theta_{t}^{\xi}\right)^{\intercal}
dW_{t},\text{ \ }t\in\lbrack0,T],\\
X_{T}^{\xi}= & \xi,\\
dY_{t}^{\xi}= & -g(t,X_{t}^{\xi}, Y_{t}^{\xi},Z_{t}^{\xi})dt+\left(Z_{t}^{\xi} \right)^{\intercal} dW_{t},\text{
\ }t\in\lbrack0,T],\\
Y_{T}^{\xi}= & \Phi(\xi),
\end{array}
\right.
\]
where $l(t,x,\theta)=-x r_{t}-\theta^{\intercal}\phi_{t} + c(x)$. 
As $l(t,0,0) \geq 0$, it ensures by the comparison theorem of BSDEs with Lipschitz generators (\cite{Karoui-Peng-Quenez-1997}, Theorem 2.2)
that if the terminal wealth $X_{T}^{\pi} \in K:=[0,+\infty)$ then the wealth process $X_{t}^{\pi} \geq 0$, $\mathbb{P}$-a.s. $t \in [0,T]$.
Therefore, the equivalent objective is
\[%
\begin{array}
[c]{rl}%
\text{minimize} & \text{ }J(\xi)=-Y_{0}^{\xi}\\
\text{subject to} & \text{ }\xi\in L_{\mathcal{F}_{T}}^{4\bar{p}^{\ast}%
}(\Omega;\mathbb{R}),\text{ \ }\xi\geq0,\text{ \ }X_{0}^{\xi}=x_{0}.
\end{array}
\]
Let $\bar{\xi}$ be an optimal terminal wealth and $\bar{X}(\cdot)$, $\bar
{Y}(\cdot)$ be the wealth process and the utility associated with $\bar{\xi}$,
respectively. According to (\ref{adj-eq}), the adjoint system is%
\[
\left\{
\begin{array}
[c]{rl}%
dp_{t}= & \left[\left(c_{x}(\bar{X}_{t})-r_{t}\right)p_{t}+g_{x}(t)q_{t}\right]dt-\phi_{t}^{\intercal}p_{t}dW_{t},\text{ \ }t\in
\lbrack0,T],\\
p_{0}= & a_{1},\\
dq_{t}= & g_{y}(t)q_{t}dt+g_{z}^{\intercal}(t)q_{t}dW_{t},\text{ \ }%
t\in\lbrack0,T],\\
q_{0}= & -a_{0},
\end{array}
\right.
\]
where $g_{w}(t)=g_{w}(t,\bar{Y}_{t},\bar{Z}_{t})$, $w=x,y,z$; $a_{0}$, $a_{1}%
\in\mathbb{R}$ with $a_{0}\geq0$ and $\left\vert a_{0}\right\vert^{2} +\left\vert
a_{1}\right\vert^{2} \neq0$. Note that in this case the mapping $h(x,y)=-y$
which leads to $q_{0}=-a_{0}$. The solution is
\[
\left\{
\begin{array}
[c]{rl}%
p_{t}= & \left[a_{1}+\int_{0}^{t} g_{x}(s)q_{s} \Lambda_{s} ds \right] \Lambda_{t}^{-1},\\
q_{t}= & -a_{0}\exp\left\{  \int_{0}^{t}\left(  g_{y}(s)-\frac{1}{2}\left\vert
g_{z}(s)\right\vert ^{2}\right)  ds+\int_{0}^{t}g_{z}^{\intercal}%
(s)dW_{s}\right\}  ,\text{ \ }t\in\lbrack0,T],
\end{array}
\right.
\]
where $\Lambda_{t} = \exp\left\{  \int_{0}^{t}\left[  r_{s}-c^{\prime}(\bar{X}_{s})+\frac{1}{2}\left\vert
\phi_{s}\right\vert ^{2}\right]  ds+\int_{0}^{t}\phi_{s}^{\intercal}dW_{s}\right\} $, $t \in [0,T]$.

Set $\Omega_{0}=\left\{  \omega\in\Omega\mid\bar{\xi}(\omega)=0\right\}  $ and
suppose Assumptions \ref{assum-1} and \ref{assum-2}. Then, by Theorem \ref{thm-SMP}, we deduce that there
exist constants $a_{0}$, $a_{1}\in\mathbb{R}$ with $a_{0}\geq0$ and
$\left\vert a_{0}\right\vert^{2} +\left\vert a_{1}\right\vert^{2} \neq0$, such that%
\begin{equation}
\label{example-SMP}
\begin{array}
[c]{l}%
p_{T}+q_{T}\Phi_{x}(\bar{\xi})\geq0, \ \mathbb{P}\text{-a.s. on }\Omega_{0},\\
p_{T}+q_{T}\Phi_{x}(\bar{\xi})=0, \ \mathbb{P}\text{-a.s.
on }\Omega\setminus\Omega_{0}.
\end{array}
\end{equation}
Once $c$, $g$, and $\Phi$ are given, we will derive the expression of the optimal
control $\bar{\xi}$. For example, we take $c(x)=\alpha x$, $\Phi(x)=\arctan(x)$, $g(t,x,y,z)=U(c(x)) - \beta y -
\frac{\gamma}{2}\left\vert z\right\vert ^{2}$, where $\alpha, \beta, \gamma>0$, and $U(\cdot)$ is a bounded utility function
that has a bounded and continuous derivative. We claim that the optimal terminal wealth can be represented as%
\begin{equation}
    \label{bar-xi-express}
    \bar{\xi}=\sqrt{\left(  -\frac{q_{T}}{p_{T}}-1\right)  ^{+}},
\end{equation}
where
\[
\begin{array}
[c]{rl}%
p_{T}= & \left[a_{1}+\alpha \int_{0}^{T} U^{\prime}(\alpha \bar{X}_{s})q_{s} \Lambda_{s} ds \right] \Lambda_{T}^{-1}  ,\\
q_{T}= & -a_{0}\exp\left\{ -\beta T - \frac{\gamma^{2}}%
{2}\int_{0}^{T}  \left\vert \bar{Z}_{s}\right\vert ^{2}  ds-\gamma\int_{0}^{T}\bar
{Z}_{s}^{\intercal}dW_{s}\right\}  
\end{array}
\]
with $\Lambda_{t} = \exp\left\{  \int_{0}^{t}\left[  r_{s}-\alpha+\frac{1}{2}\left\vert
\phi_{s}\right\vert ^{2}\right]  ds+\int_{0}^{t}\phi_{s}^{\intercal}dW_{s}\right\} $, $t \in [0,T]$.
We provide a sketch of the proof.

%\begin{proof}
\textbf{Case 1: }$a_{0}>0$. In this case, we deduce that $q_{T}<0$. Hence, from (\ref{example-SMP}), on the one hand we have, 
on $\Omega_{0}$, $\mathbb{P}$-a.s.,
\[
p_{T}+q_{T}\geq0\Longrightarrow p_{T}\geq-q_{T}>0\Longrightarrow\left\{
\begin{array}
[c]{l}%
-\frac{q_{T}}{p_{T}}-1\leq0,\\
a_{1}+\alpha \int_{0}^{T} U^{\prime}(\alpha \bar{X}_{s})q_{s} \Lambda_{s} ds > 0.
\end{array}
\right.
\]
On the other hand, on $\Omega\setminus\Omega_{0}$, $\mathbb{P}$-a.s.,
\[
p_{T}+q_{T}\Phi_{x}(\bar{\xi})=0\Longrightarrow-q_{T}>p_{T}=-q_{T}\Phi
_{x}(\bar{\xi})>0\Longrightarrow\left\{
\begin{array}
[c]{l}%
\left\vert \bar{\xi}\right\vert ^{2}=-\frac{q_{T}}{p_{T}}-1>0,\\
a_{1}+\alpha \int_{0}^{T} U^{\prime}(\alpha \bar{X}_{s})q_{s} \Lambda_{s} ds > 0.
\end{array}
\right.
\]

\textbf{Case 2: }$a_{0}=0$. In this case, we deduce that $a_{1}\neq0$ and 
$q_{t}=0$, $t \in [0,T]$. Hence, from (\ref{example-SMP}), we have $p_{T}\geq0$ on 
$\Omega_{0}$ and $p_{T}=0$ on $\Omega\setminus\Omega_{0}$, $\mathbb{P}$-a.s.. But $a_{1}\neq0$ implies that $p_{T}>0$. 
So we deduce that $\bar{\xi}=0$ $\mathbb{P}$-a.s. and $a_{1}>0$. 

In summary, for both cases, we have (\ref{bar-xi-express}).
%\end{proof}

\begin{remark}
In view of the stochastic differential utility, the above example is closely related to the robust expected utility model studied in \cite{Quenez-2003}. The generator $g(t,x,y,z)=U(c(x))-\beta
y-\frac{\gamma}{2}\left\vert z\right\vert ^{2}$ can be interpreted as
an intertemporal aggregator where $U(c(x))-\beta y$ corresponds to the standard expected additive utility in continuous-time, and 
$\gamma>0$ is the risk-averse parameter which reflects the issue of robustness in portfolio decision (see \cite{Skiadas-2001} for more details).
\end{remark}


\begin{thebibliography}{99}
    
    \bibitem {Barr-ElKar}P. Barrieu and N. El Karoui, Monotone stability of
    quadratic semimartingales with applications to unbounded general quadratic
    BSDEs. \emph{Ann. Probab.} \textbf{41} (2013) 1831-1863.
    
    \bibitem {Pliska}T. R. Bielecki, H. Jin, S. R. Pliska and X. Zhou, Continuous time mean variance portfolio selection with bankruptcy
    prohibition. \emph{Math. Finance} \textbf{15} (2005) 213-244.
    
    \bibitem {Billingsley}P. Billingsley, \emph{Probability and Measure, Second Edition}.
    John Wiley \& and Sons, Canada (1986).
    
    \bibitem{Bismut} J.-M. Bismut, Linear quadratic optimal stochastic control with random coefficients. \emph{SIAM J. Control Optim.} \textbf{14}
    (1976) 419-444.
    
    \bibitem{Matoussi2011} W. Faidi, A. Matoussi, and M. Mnif, Maximization of Recursive Utilities: A Dynamic Maximum Principle Approach.
    \emph{SIAM J. Financial Math.} \textbf{2} (2011) 1014-1041.
    
    \bibitem {Briand-SL-BSDE}P. Briand and F. Confortola, BSDEs with stochastic
    Lipschitz condition and quadratic PDEs in Hilbert spaces. \emph{Stochastic Process. Appl.} \textbf{118} (2008) 818-838.
    
    \bibitem {HuBSDEquad08}P. Briand and Y. Hu, Quadratic BSDEs with convex
    generators and unbounded terminal conditions. \emph{Probab. Theory Related Fields} \textbf{141} (2008) 543-567.
    
    \bibitem {Delbaen02}F. Delbaen, P. Grandits, T. Rheinl\"{a}nder, D. Samperi,
    M. Schweizer and C. Stricker, Exponential hedging and entropic penalties.
    \emph{Math. Finance} \textbf{12} (2002) 99-123.
    
    \bibitem {Delbaen-Hu-Richou}F. Delbaen, Y. Hu, A. Richou, On the uniqueness of solutions to quadratic BSDEs with convex generators and
    unbounded terminal conditions. \emph{Ann. Inst. Henri Poincar\'{e} Probab. Stat.} \textbf{47} (2011) 559-574.
    
    \bibitem {Dudley}R. M. Dudley, \emph{Real Analysis and Probability}. The Press
    Syndicate of the University of Cambridge, United Kingdom (2004).
    
    \bibitem {Ekeland-1974}I. Ekeland, On the variational principle. \emph{J. Math. Anal. Appl.} \textbf{47} (1974) 324-353.
    
    \bibitem {Karoui-Hamadene}N. El Karoui, S. Hamad\`{e}ne, BSDEs and risk-sensitive control, zero-sumand nonzero-sum game problems of
    stochastic functional differential equations. \emph{Stochastic Process. Appl.} \textbf{107} (2003) 145-169.
    
    \bibitem {Karoui-Peng-Quenez-1997}N. El Karoui, S. Peng, M. C. Quenez, Backward stochastic differential equations in finance. \emph{Math. Finance} \textbf{7}
    (1997) 1-71.
    
    \bibitem {Karoui-Peng-Quenez-2001}N. El Karoui, S. Peng, M. C. Quenez, A dynamic maximum principle for the optimization of recursive utilities under constraints. 
    \emph{Ann. Appl. Probab.} \textbf{11} (2001) 664-693.
    
    \bibitem {Gal'Chuk1978}I. L. Gal'Chuk, Existence and uniqueness of a solution
    for stochastic equations with respect to semimartingales. \emph{Theory Probab. Appl.} \textbf{23} (1978) 751-763.
    
    \bibitem {HuJiXue18}M. Hu, S. Ji and X. Xue, A global stochastic maximum
    principle for fully coupled forward-backward stochastic systems. \emph{SIAM J. Control Optim.} \textbf{56} (2018) 4309-4335.
    
    \bibitem {Hu-Im-Mull}Y. Hu, P. Imkeller and M. M\"{u}ller, Utility
    maximization in incomplete markets. \emph{Ann. Appl. Probab.} \textbf{15} (2005) 1691-1712.
    
    \bibitem {HuYing-Peng95}Y. Hu and S. Peng, Solution of forward-backward
    stochastic differential equations.  \emph{Probab. Theory Related Fields} \textbf{103} (1995) 273-283.
    
    \bibitem {Hu-Tang2016}Y. Hu, S. Tang, Multi-dimensional backward stochastic differential equations of diagonally quadratic generators. 
    \emph{Stochastic Process. Appl.} \textbf{126} (2016) 1066-1086.
    
    \bibitem {Ji-Peng-2008}S. Ji and S. Peng, Terminal perturbation method for the backward
    approach to continuous time mean-variance portfolio selection, \emph{Stochastic Process. Appl.} \textbf{118} (2008) 952-967.
    
    \bibitem {Ji-Zhou-2006}S. Ji and X. Zhou, A maximum principle for stochastic optimal optimal control with terminal state
    constraints and its applications. \emph{Commun. Inf. Syst.} \textbf{6} (2006) 321-337.
    
    \bibitem {Ji-Zhou-2010}S. Ji and X. Zhou, A generalized Neyman-Pearson lemma
    for g-probabilities. \emph{Probab. Theory Related Fields} \textbf{148} (2010) 645-669.
    
    \bibitem {Quenez-2003}A. Lazrak and M.C. Quenez, A generalized stochastic differential utility. 
    \emph{Math. Oper. Res.} \textbf{28} (2003) 154-180.
    
    \bibitem {Kazamaki}N. Kazamaki, \emph{Continuous exponential martingales and BMO}.
    Springer-Verlag Berlin Heidelberg, (1994).
    
    \bibitem {Kobylanski}M. Kobylanski, Backward stochastic differential equations
    and partial differential equations with quadratic growth. \emph{Ann. Probab.} \textbf{28} (2000) 558-602.
    
    \bibitem {Lim-Zhou} A. E. B. Lim and X. Zhou, A New Risk-Sensitive Maximum Principle. \emph{IEEE Trans. Automat. Control} \textbf{50} (2005) 958-966.
    
    \bibitem {Ma-Yong-Protter}J. Ma, P. Protter and J. Yong, Solving
    forward-backward stochastic differential equations explicitly - a 4 step
    scheme. \emph{Probab. Theory Related Fields} \textbf{98} (1994) 339-359.
    
    \bibitem {Ma-WZZ}J. Ma, Z. Wu, D. Zhang and J. Zhang, On well-posedness of
    forward-backward SDEs - a unified approach. \emph{Ann. Appl. Probab.} \textbf{25} (2015) 2168-2214.
    
    \bibitem {Moon2021}J. Moon, Generalized risk-sensitive optimal control and Hamilton-Jacobi-Bellman equation. 
    \emph{IEEE Trans. Automat. Control} \textbf{66} (2021) 2319-2325.
    
    \bibitem {Peng90}S. Peng, A general stochastic maximum principle for optimal
    control problems. \emph{SIAM J. Control Optim.} \textbf{28} (1990) 966-979.
    
    \bibitem {Peng93}S. Peng, Backward stochastic differential equations and
    applications to optimal control. \emph{Appl. Math. Optim.} \textbf{27}
    (1993) 125-144.
    
    \bibitem {Skiadas-2001}C. Skiadas, Robust control and recursive utility. {Finance Stoch.} \textbf{7} (2003) 475-489.
    
    \bibitem {Tang03}S. Tang, General linear quadratic optimal stochastic control
    problems with random coefficients: linear stochastic Hamilton systems and
    backward stochastic Riccati equations. \emph{SIAM J. Control Optim.} \textbf{42} (2003) 53-75.
    
    \bibitem {Tevzadze}R. Tevzadze, Solvability of backward stochastic differential equations with quadratic growth, 
    \emph{Stoch. Process. Appl.} \textbf{118} (2008) 503-515.
    
    \bibitem{Wei2013} Q. Wei, A maximum principle for fully coupled forward-backward stochastic control systems with terminal state constraints. 
    \emph{J. Math. Anal. Appl.} \textbf{407} (2013) 200-210.
    
    \bibitem{Wei2016} Q. Wei, Stochastic maximum principle for mean-field forward-backward stochastic control system with terminal state constraints. 
    \emph{Sci. China Math.} \textbf{59} (2016) 809-822.
    
    \bibitem {Wu2013}Z. Wu, A general maximum principle for optimal control of forward-backward stochastic systems.
    \emph{Automatica} \textbf{49} (2013) 1473-1480.
    
    \bibitem {Xing18}H. Xing and G. \v{Z}itkovi\'{c}, A class of globally solvable
    Markovian quadratic BSDE systems and applications. \emph{Ann. Probab.} \textbf{46} (2018) 491-550.
    
    \bibitem {Yong2010}J. Yong, Optimality variational principle for controlled
    forward-backward stochastic differential equations with mixed initial-terminal conditions.
    \emph{SIAM J. Control Optim.} \textbf{48} (2010) 4119-4156.
    
\end{thebibliography}
\end{document}